\documentclass[12 pt]{amsart}
\usepackage{amsmath,times,epsfig,amssymb,amsbsy,amscd,amsfonts,amstext,color,bm}
\usepackage[arrow,matrix]{xy}

\usepackage{graphicx}

\theoremstyle{plain}
\newtheorem{theorem}{Theorem}[section]
\newtheorem{corollary}[theorem]{Corollary}
\newtheorem{lemma}[theorem]{Lemma}

\newtheorem{proposition}[theorem]{Proposition}
\newtheorem{conjecture}[theorem]{Conjecture}

\theoremstyle{definition}
\newtheorem{definition}[theorem]{Definition}
\newtheorem{example}[theorem]{Example}

\theoremstyle{remark}
\newtheorem{remark}[theorem]{Remark}

\newcommand{\bA}{\ensuremath{\mathbb{A}}}

\newcommand{\bC}{\ensuremath{\mathbb{C}}}

\newcommand{\bN}{\ensuremath{\mathbb{N}}}

\newcommand{\bQ}{\ensuremath{\mathbb{Q}}}
\newcommand{\bR}{\ensuremath{\mathbb{R}}}

\newcommand{\bZ}{\ensuremath{\mathbb{Z}}}

\newcommand{\scP}{\ensuremath{\mathcal{P}}}

\newcommand{\frakS}{\ensuremath{\mathfrak{S}}}

\newcommand{\A}{\mathcal{A}}
\newcommand{\Z}{\mathbb{Z}}

\newcommand{\asc}{\operatorname{asc}}
\newcommand{\cdes}{\operatorname{cdes}}
\newcommand{\dsc}{\operatorname{dsc}}
\newcommand{\vol}{\operatorname{vol}}
\newcommand{\eul}{\operatorname{A}}
\newcommand{\Hom}{\operatorname{Hom}}
\newcommand{\geneul}{\operatorname{R}}
\newcommand{\ehr}{\operatorname{L}}
\newcommand{\alcov}{{A^{\circ}}}
\newcommand{\alcovxi}{{A^{\circ}_{\xi}}}
\newcommand{\baralcov}{\overline{\alcov}}
\newcommand{\semialcov}{{A^{\diamondsuit}_{\xi}}}
\newcommand{\parallelp}{P^\diamondsuit}

\newcommand{\cwtl}{Z}
\newcommand{\crtl}{\check{Q}}

\newcommand{\quasi}{\operatorname{quasi}}
\newcommand{\aff}{\operatorname{aff}}
\newcommand{\rad}{\operatorname{rad}}

\renewcommand{\Re}{\operatorname{Re}}

\begin{document}

\title[Worpitzky partition]{Worpitzky partitions for root systems and 
characteristic quasi-polynomials}

\begin{abstract}
For a given irreducible root system, 
we introduce a partition of (coweight) lattice points 
inside the dilated fundamental parallelepiped into those of 
partially closed simplices. 
This partition can be considered 
as a generalization and a lattice points interpretation of 
the classical formula of Worpitzky. 

This partition, and the generalized Eulerian polynomial, 
recently introduced by Lam and Postnikov, can be used to 
describe the characteristic (quasi)polynomials of Shi and 
Linial arrangements. 
As an application, we prove that 
the characteristic quasi-polynomial of the Shi arrangement turns 
out to be a polynomial. 
We also present several results on the location 
of zeros of characteristic polynomials, 
related to a conjecture of Postnikov and Stanley. 
In particular, we verify the ``functional equation'' of the 
characteristic polynomial of the Linial arrangement for any 
root system, and give partial affirmative results on 
``Riemann hypothesis'' 
for the root systems of type $E_6, E_7, E_8$, and $F_4$. 
\end{abstract}

\author{Masahiko Yoshinaga}
\address{Masahiko Yoshinaga, Department of Mathematics, Hokkaido University, Kita 10, Nishi 8, Kita-Ku, Sapporo 060-0810, Japan.}
\email{yoshinaga@math.sci.hokudai.ac.jp}




\date{\today}
\maketitle

\tableofcontents

\section{Introduction}
\label{sec:intro}

\subsection{Arrangements of hyperplanes}

Let $\A=\{H_1, \dots, H_n\}$ be an arrangement of affine hyperplanes in 
a vector space $V$. The intersection poset of $\A$ is the set 
$L(\A)=\{\cap S\mid S\subset \A\}$ of intersections of $\A$. 
The intersection poset $L(\A)$ is partially ordered by 
reverse inclusion, which has a unique minimal element 
$\hat{0}=V$. The arrangement $\A$ is called essential if 
the maximal elements of $L(\A)$ are $0$-dimensional subspaces. 
The characteristic polynomial of $\A$ is 
defined by 
\[
\chi(A, q)=\sum_{X\in L(\A)}\mu(X)q^{\dim X}, 
\]
where $\mu$ is the M\"obius function on $L(\A)$, defined by 
\[
\mu(X)=
\left\{
\begin{array}{ll}
1, & \mbox{ if } X=\hat{0}\\
-\sum_{Y<X}\mu(Y), & \mbox{ otherwise.} 
\end{array}
\right.
\]
The characteristic polynomial $\chi(\A, q)$ of $\A$ is one of the most 
fundamental invariants. Indeed, $\chi(\A, q)$ captures 
combinatorial and topological properties of $\A$ as follows. 

\begin{theorem}
(i) (Zaslavsky \cite{zas-face}). 
Suppose that $V$ is a real vector space. 
Then the number of connected components of $V\smallsetminus\cup\A$ is 
equal to $(-1)^\ell\chi(\A, -1)$. If $\A$ is essential, then the 
number of bounded connected components of $V\smallsetminus\cup\A$ is 
equal to $(-1)^\ell\chi(\A, 1)$. 

(ii) (Orlik and Solomon \cite{os}). Suppose that $V$ is a complex 
vector space. Then the Poincar\'e polynomial of $V\smallsetminus\cup\A$ 
is equal to $(-t)^\ell\chi(\A, -\frac{1}{t})$. 
\end{theorem}

\subsection{Main Results}

Let $\Phi$ be a root system of rank $\ell$, with 
exponents $e_1, \dots, e_\ell$ and Coxeter number $h$. 
Fix a set of positive roots $\Phi^+\subset\Phi$. 
The structures of 
truncated affine Weyl arrangements 
\[
\A_{\Phi}^{[a, b]}=
\{H_{\alpha, k}\mid\alpha\in\Phi^+, a\leq k\leq b\}
\]
(see also \S\ref{subsec:rs} for the notation) 
have been intensively studied, because of their 
intriguing combinatorial properties 
\cite{shi-kl, ath-survey, st-lect}. 
The characteristic polynomial of the Coxeter arrangement 
$\A_{\Phi}=\A_{\Phi}^{[{0,0}]}$ was computed in \cite{bri-tress}. 
Later, this was generalized to 
the extended Catalan arrangement $\A_{\Phi}^{[-k,k]}$ and 
the extended Shi arrangement $\A_{\Phi}^{[1-k,k]}$.  
The characteristic polynomial of these arrangements factor as follows 
\[
\begin{split}
\chi(\A_{\Phi}^{[-k,k]},t)&=\prod_{i=1}^\ell(t-e_i-kh), \\
\chi(\A_{\Phi}^{[1-k,k]},t)&=(t-kh)^\ell, 
\end{split}
\]
(\cite{ede-rei, ath-adv, ath-gen, yos-char}). 
For other parameters $a\leq b$, e.g., the Linial arrangement 
$\A_{\Phi}^{[1, n]}$, 
the characteristic polynomial $\chi(\A_{\Phi}^{[a, b]}, t)$ does 
not factor in general. However there are a number of 
beautiful conjectures concerning $\chi(\A_{\Phi}^{[a, b]}, t)$. 
Among others, Postnikov and Stanley \cite{ps-def} conjectured that 
\begin{itemize}
\item[(a)] 
$\chi(\A_{\Phi}^{[1-k, n+k]}, t)=\chi(\A_{\Phi}^{[1,n]}, t-kh)$ 
(``$h$-shift reduction''). 
\item[(b)] 
$\chi(\A_{\Phi}^{[1, n]}, nh-t)=(-1)^\ell\chi(\A_{\Phi}^{[1, n]}, t)$ 
(``Functional equation''). 
\item[(c)] 
All the roots of the polynomial $\chi(\A_{\Phi}^{[1, n]}, t)$ have 
the same real part $\frac{nh}{2}$ 
(``Riemann hypothesis''). 
\end{itemize}
Postnikov and Stanley verified these assertions for $\Phi=A_\ell$ in 
\cite{ps-def}. 
Later, Athanasiadis gave proofs for $\Phi=A_\ell, B_\ell, C_\ell, 
D_\ell$ in \cite{ath-adv, ath-lin}. 

Recently, Kamiya, Takemura and Terao \cite{ktt-quasi} 
introduced the notion of 
the characteristic quasi-polynomial for an 
arrangement $\A$ defined over $\bQ$. 
The characteristic quasi-polynomial $\chi_{\quasi}(\A, t)$ is a 
periodic polynomial (see \S\ref{subsec:arr} for details) that 
may be considered as a refinement of the characteristic polynomial 
$\chi(\A, t)$. 

Our main result concerns the characteristic quasi-polynomial 
for $\A_{\Phi}^{[a, b]}$: ``$h$-shift reduction'' and 
the ``functional equation'' hold at the level of characteristic 
quasi-polynomials. 

\begin{theorem}
Let $\Phi$ be an arbitrary irreducible root system. 
\begin{itemize}
\item[(i)] 
The characteristic quasi-polynomial of the 
extended Shi arrangement is a polynomial, 
$\chi_{\quasi}(\A_{\Phi}^{[1-k,k]}, t)=(t-kh)^\ell$ 
(Theorem \ref{thm:shi}). 
\item[(ii)] 
The characteristic quasi-polynomial satisfies ``$h$-shift reduction'' 
$\chi_{\quasi}(\A_{\Phi}^{[1-k, n+k]}, t)=
\chi_{\quasi}(\A_{\Phi}^{[1,n]}, t-kh)$ 
(Theorem \ref{thm:extendedlinial}). In particular, this holds for 
the characteristic polynomial (Corollary \ref{cor:shift}). 
\item[(iii)] 
The characteristic quasi-polynomial satisfies 
``Functional equation'' 
$\chi_{\quasi}(\A_{\Phi}^{[1, n]}, nh-t)=
(-1)^\ell\chi_{\quasi}(\A_{\Phi}^{[1, n]}, t)$ 
(Theorem \ref{thm:dualityquasichar}). 
In particular, this holds for the characteristic polynomial 
(Corollary \ref{cor:generalfuncteq}). 
\item[(iv)] 
Suppose $\Phi\in\{E_6, E_7, E_8, F_4\}$. 
Let $\widetilde{n}$ be the period of the Ehrhart quasi-polynomial of the 
fundamental alcove (see \S \ref{subsec:ehralcov} and Table \ref{fig:table}). 
If $n\equiv -1\mod \rad(\widetilde{n})$, then 
the ``Riemann hypothesis'' holds for $\A_{\Phi}^{[1,n]}$ 
(Theorem \ref{thm:rh}). 
\end{itemize}
\end{theorem}

\subsection{Outline of the proof}

We follow the strategy adopted in 
\cite{bl-sa, ath-adv, ath-gen, ktt-quasi} for the computation 
of $\chi_{\quasi}(\A_{\Phi}^{}, q)$, where $\A_{\Phi}:=\A_{\Phi}^{[{0,0}]}$ 
is the Coxeter arrangement 
(see \S \ref{subsec:charquasi-poly}). The idea is 
to relate the characteristic quasi-polynomial to the 
Ehrhart quasi-polynomial $\ehr_{\alcov}(q)$ of the fundamental alcove 
$\alcov$. 
Consider the associated hyperplane arrangement $\overline{\A}$ 
in the quotient 
$\cwtl(\Phi)/q\cwtl(\Phi)$, where $\cwtl(\Phi)$ is the coweight lattice. 
Then, by definition, $\chi_{\quasi}(\A_{\Phi}^{}, q)$ is 
the number of points in the complement of $\overline{\A}$, for $q\gg 0$. 
If we define $\parallelp=\sum_{k=1}^\ell(0,1]\varpi_i$, 
(where $\varpi_i^\lor$ is the basis dual to the simple basis), then 
there is a bijective correspondence between 
the points in $\cwtl(\Phi)/q\cwtl(\Phi)$
and the lattice points in the dilated parallelepiped $q\parallelp$. 
The parallelepiped $\parallelp$ is dissected by the affine Weyl arrangement 
into open simplices (alcoves). 
Thus $\chi_{\quasi}(\A_{\Phi}^{},q)$ can be expressed as the sum of 
Ehrhart quasi-polynomials of these alcoves. Since $\A_{\Phi}^{}$ 
is Weyl group invariant, the above dissection is into 
simplices of the same size (Figure \ref{fig:b2coxeter}), 
which yields the simple formula 
\begin{equation}
\label{eq:dissect}
\chi_{\quasi}(\A_{\Phi}^{}, q)
=
\frac{f}{|W|}\cdot 
\ehr_{\alcov}(q)
\end{equation}
(see Corollary \ref{cor:fullremove}, Corollary \ref{cor:duality} and 
Proposition \ref{prop:charquasicoxeter}). The formula (\ref{eq:dissect}) 
first appeared explicitly in \cite{bl-sa}, where it was proved using 
the classification of root systems. A case free proof was given 
in \cite{ath-adv}. The argument was later extended to the case 
$\A_{\Phi}^{[-m,m]}$ in \cite{ath-gen}. 

If we apply the same strategy for the case of Shi and Linial 
arrangements, then 
$\chi_{\quasi}(\A_{\Phi}^{[a, b]}, q)$ can again be expressed 
as the sum of Ehrhart quasi-polynomials. However the sizes of 
simplices are no longer uniform (see Figure \ref{fig:shilinial}). 
This difficulty can be overcome by looking at a 
disjoint partition of $\parallelp$ into partially closed alcoves 
\begin{equation}
\parallelp=\bigsqcup_{\xi\in\Xi}\semialcov, 
\end{equation}
(see \S \ref{subsec:partition} for details). Then obviously 
we have a partition of lattice points 
\begin{equation}
\label{eq:introwp}
q\parallelp\cap\cwtl(\Phi)=
\bigsqcup_{\xi\in\Xi}(q\semialcov\cap\cwtl(\Phi)), 
\end{equation}
which we will call a Worpitzky partition. 
The number of lattice points contained in $q\semialcov$ is 
expressed as 
\begin{equation}
\ehr_{\semialcov}(q)=\ehr_{\baralcov}(q-\asc(\alcovxi)), 
\end{equation}
(Lemma \ref{lem:alcovediamond}), where $\asc(\alcovxi)$ is 
a certain integer (Definition \ref{def:asc} and (\ref{eq:ascgen})). 
The key result (Theorem \ref{thm:eulerianparallelp}) 
in the proof of our main results is 
that the distribution of the quantity $\asc(\alcovxi)$ is given by 
the generalized 
Eulerian polynomial $\geneul_{\Phi}(t)$ (Definition \ref{def:geneul}) 
introduced by Lam and Postnikov \cite{lp-alc2}. 
Using the shift operator $S$ (\S \ref{subsec:rh}), the partition 
(\ref{eq:introwp}) implies the formula, 
\begin{equation}
\label{eq:introsum}
q^\ell=
(\geneul_{\Phi}(S)\ehr_{\baralcov})(q). 
\end{equation}
In the case $\Phi=A_\ell$, the polynomial 
$\geneul_{\Phi}(t)$ is equal to the classical Eulerian polynomial. Then 
the above formula (\ref{eq:introsum}) is known as the 
Worpitzky identity \cite{wor, comtet}. Hence (\ref{eq:introsum}) can 
be considered as a generalization of Worpitzky identity and 
(\ref{eq:introwp}) as its lattice points interpretation. 

Using these results, the characteristic quasi-polynomials for 
Shi and Linial arrangements have expressions similar to 
the Worpitzky identity (\ref{eq:introsum}). We have 
\begin{equation}
\begin{split}
\chi_{\quasi}(\A_{\Phi}^{[1-k, k]}, q)
=&
(S^{kh}\geneul_{\Phi}(S)\ehr_{\baralcov})(q)=(q-kh)^\ell, \\
\chi_{\quasi}(\A_{\Phi}^{[1-k, n+k]}, q)
=&
(S^{kh}\geneul_{\Phi}(S^{n+1})\ehr_{\baralcov})(q), 
\end{split}
\end{equation}
(Theorem \ref{thm:shi}, Theorem \ref{thm:extendedlinial}). 
Using these expressions, the functional equation is obtained from 
the duality of the generalized Eulerian polynomial 
\begin{equation}
t^h\geneul_{\Phi}(\frac{1}{t})=\geneul_{\Phi}(t), 
\end{equation}
(Proposition \ref{prop:properties}). 

If $n\equiv -1\mod\rad(\widetilde{n})$, then $1+n$ is divisible by 
$\rad(\widetilde{n})$. Hence 
$\gcd(q, \widetilde{n})=1$ implies $\gcd(q-k(n+1), \widetilde{n})=1$ 
for $k\in\bZ$. 
This enables us to simplify the expression of the characteristic 
polynomial $\chi(\A_{\Phi}^{[1, n]}, t)$. Using techniques similar 
to those in Postnikov, Stanley and Athanasiadis \cite{ps-def, ath-lin}, we 
can verify the ``Riemann hypothesis'' for such parameters $n$. 

The paper is organized as follows. \S \ref{sec:background} contains 
background materials on root systems, characteristic quasi-polynomials 
and the Eulerian polynomial. The partition of the fundamental 
parallelepiped, which will play an important role later, is introduced 
in \S \ref{subsec:partition} (Definition \ref{def:semialcov}). 
In \S \ref{sec:ehr} the relation between the Ehrhart quasi-polynomial 
of the fundamental alcove and the characteristic quasi-polynomial is 
discussed. 
In \S \ref{sec:geul} we first summarize basic properties of 
the generalized Eulerian polynomial $\geneul_{\Phi}(t)$ 
introduced by Lam and Postnikov \cite{lp-alc2}. Then we introduce 
Worpitzky partitions of the lattice points which provide a 
Worpitzky-type identity (Theorem \ref{thm:worpart}). We also 
give an explicit example of the 
Worpitzky partition for $\Phi=B_2$. 
In \S \ref{sec:main}, we obtain formulae for characteristic 
quasi-polynomials by modifying the Worpitzky-type identity. 
Using these formulae, we prove our main results.

\section{Background}
\label{sec:background}

\subsection{Quasi-polynomials with $\gcd$-property} 
\label{subsec:quasi-poly}

A function $f:\bZ\longrightarrow\bZ$ is called a 
\emph{quasi-polynomial} if there exist $\widetilde{n}>0$ and 
polynomials $g_1(t), g_2(t),\dots, g_{\widetilde{n}}(t)\in
\bZ[t]$ such that 
\[
f(q)=g_r(q), \mbox{ if }q\equiv r \mod \widetilde{n}, 
\]
($1\leq r\leq\widetilde{n}$). The minimal such $\widetilde{n}$ 
is called the period of the quasi-polynomial $f$. 

Moreover, the function 
$f:\bZ\longrightarrow\bZ$ is said to be a 
\emph{quasi-polynomial with $\gcd$-property} if the 
polynomial $g_r(t)$ depends on $r$ only through 
$\gcd(r, \widetilde{n})$. In other words, $g_{r_1}(t)=g_{r_2}(t)$ 
if $\gcd(r_1, \widetilde{n})=\gcd(r_2, \widetilde{n})$. 

\subsection{Arrangements and characteristic quasi-polynomials}
\label{subsec:arr}

Let $L\simeq\bZ^\ell$ be a lattice and $L^\lor=\Hom_{\bZ}(L, \bZ)$ 
be the dual lattice. Given $\alpha_1, \dots, \alpha_n\in L^\lor$ and 
integers $k_1, \dots, k_n\in\bZ$, we can associate a 
hyperplane arrangement $\A=\{H_1, \dots, H_n\}$ in 
$\bR^\ell\simeq L\otimes_{\bA}\bR$, with $H_i=\{x\in L\otimes\bR\mid
\alpha_i(x)=k_i\}$. For a positive integer $q>0$, define 
\begin{equation}
\label{eq:modq}
M(\A; q):=
\{\overline{x}\in L/qL\mid
\forall i, \alpha(x)\not\equiv k_i \mod q \}. 
\end{equation}
Kamiya, Takemura and Terao proved the following. 

\begin{theorem}
\label{thm:ktt}
(\cite{ktt-cent, ktt-noncent}) There exist $q_0>0$ and a quasi-polynomial 
$\chi_{\quasi}(\A, t)$ with $\gcd$-property such that 
$\# M(\A, q)=\chi_{\quasi}(\A, q)$ for $q>q_0$. 
\end{theorem}
More precisely, there exists a period $\widetilde{n}$ and 
a polynomial 
$g_d(t)\in\bZ[t]$ for each divisor $d|\widetilde{n}$ such 
that 
\[
\# M(\A; q)=g_d(q), 
\]
for $q>q_0$, 
where $d=\gcd(\widetilde{n}, q)$. 

One of the most important invariants of a hyperplane arrangement $\A$ 
is the characteristic polynomial $\chi(\A, t)\in\bA[t]$ (see 
\cite{ot} for the definition and basic properties). 
The characteristic polynomial is one of the 
polynomials given by Theorem \ref{thm:ktt} 
(see also \cite[Theorem 2.1]{ath-lin}), 
specifically, 
\begin{equation}
\label{eq:chig1}
\chi(\A, t)=g_1(t). 
\end{equation}

\subsection{Root systems}
\label{subsec:rs}

Let $V=\bR^\ell$ be the Euclidean space with inner product 
$(\cdot, \cdot)$. Let $\Phi\subset V$ be an irreducible 
root system with exponents $e_1, \dots, e_\ell$, 
Coxeter number $h$ and Weyl group $W$. 
For any integer $k\in\bZ$ and $\alpha\in\Phi^+$, 
the affine hyperplane $H_{\alpha, k}$ is defined by 
\begin{equation}
\label{eq:affinehyperp}
H_{\alpha, k}=\{x\in V\mid 
(\alpha, x)=k\}. 
\end{equation}

Fix a positive system $\Phi^+\subset \Phi$ and 
the set of simple roots $\Delta=\{\alpha_1, \dots, \alpha_\ell\}
\subset\Phi^+$. The highest root, denoted by 
$\widetilde{\alpha}\in\Phi^+$, can be expressed as a linear 
combination $\widetilde{\alpha}=\sum_{i=1}^\ell c_i\alpha_i$ 
($c_i\in\bZ_{>0}$). We also set $\alpha_0:=-\widetilde{\alpha}$ and 
$c_0:=1$. Then we have the linear relation 
\begin{equation}
\label{eq:linrel}
c_0\alpha_0+c_1\alpha+\dots+c_\ell\alpha_\ell=0. 
\end{equation}
The coweight lattice $\cwtl(\Phi)$ and the coroot lattice $\crtl(\Phi)$ 
are defined as follows. 
\[
\begin{split}
\cwtl(\Phi)
&=
\{x\in V\mid (\alpha_i, x)\in\bZ, \alpha_i\in\Delta\}, \\
\crtl(\Phi)
&=
\sum_{\alpha\in\Phi}\bZ\cdot
\frac{2\alpha}{(\alpha, \alpha)}. 
\end{split}
\]
The coroot lattice $\crtl(\Phi)$ is a finite index subgroup 
of the coweight lattice $\cwtl(\Phi)$. The index 
$\#\frac{\cwtl(\Phi)}{\crtl(\Phi)}=f$ is called the 
\emph{index of connection}. 

Let $\varpi_i^\lor\in\cwtl(\Phi)$ be the dual basis to the 
simple roots $\alpha_1, \dots, \alpha_\ell$, that is, 
$(\alpha_i, \varpi_j^\lor)=\delta_{ij}$. Then $\cwtl(\Phi)$ is 
a free abelian group generated by 
$\varpi_1^\lor, \dots, \varpi_\ell^\lor$. We also have 
$c_i=(\varpi_i^\lor, \widetilde{\alpha})$. 

A connected component of $V\smallsetminus 
\bigcup\limits_{\substack{\alpha\in\Phi^+\\ k\in\bZ}}H_{\alpha, k}$ 
is called an \emph{alcove}. 
Let us define the fundamental alcove $\alcov$ by 
\[
\begin{split}
\alcov
=&
\left\{
x\in V
\left|
\begin{array}{ll}
(\alpha_i, x)>0,&(1\leq i\leq \ell)\\
(\widetilde{\alpha}, x)<1&
\end{array}
\right.
\right\}
\\
=&
\left\{
x\in V
\left|
\begin{array}{ll}
(\alpha_i, x)>0,&(1\leq i\leq \ell)\\
(\alpha_0, x)>-1&
\end{array}
\right.
\right\}. 
\end{split}
\]
The closure 
$\baralcov=\{x\in V\mid (\alpha_i, x)\geq 0\ 
(1\leq i\leq \ell),\ (\widetilde{\alpha}, x)\leq 1\}$ 
is the convex hull of $0, \frac{\varpi_1^\lor}{c_1}, \dots, 
\frac{\varpi_\ell^\lor}{c_\ell}\in V$. 
The closed alcove $\baralcov$ is a simplex. 
The supporting hyperplanes of facets of $\baralcov$ are 
$H_{\alpha_1, 0}, \dots, H_{\alpha_\ell, 0}, H_{\widetilde{\alpha}, 1}$. 
We note that 
$\baralcov$ is a fundamental domain of the affine Weyl group 
$W_{\aff}=W\ltimes\crtl(\Phi)$.

Let $\parallelp$ denote the fundamental domain 
of the coweight lattice $\cwtl(\Phi)$ defined by 
\begin{equation}
\label{eq:defparallelp}
\begin{split}
\parallelp
&=
\sum_{i=1}^\ell (0,1]\varpi_i^\lor\\
&=
\{x\in V\mid 0<(\alpha_i, x)\leq 1, i=1, \dots, \ell\}. 
\end{split}
\end{equation}
Here we summarize without proofs some useful facts on root systems \cite{hum}. 
\begin{proposition}
\label{prop:rootsys}
\begin{itemize}
\item[(i)] 
$c_0+c_1+\dots+c_\ell=h$. 
\item[(ii)] 
$\frac{|W|}{f}=\frac{\vol(\parallelp)}{\vol(\alcov)}=
l!\cdot c_1\cdot c_2 \cdots c_\ell$. 
\item[(iii)] 
$|\Phi^+|=\frac{\ell h}{2}$. 
\end{itemize}
\end{proposition}

\subsection{Partition of the fundamental parallelepiped}
\label{subsec:partition}

Let us consider the set of alcoves contained in $\parallelp$, denoted by 
$\{\alcovxi\mid \xi\in\Xi\}$, where $\Xi$ is a finite set 
with $|\Xi|=\frac{|W|}{f}$ (by Proposition \ref{prop:rootsys} (ii)). 
In other words, 
\begin{equation}
\label{eq:alcoves}
\parallelp\smallsetminus\bigcup_{\alpha\in\Phi^+, k\in\bZ}
H_{\alpha, k}=
\bigsqcup_{\xi\in\Xi}A_{\xi}^{\circ}. 
\end{equation}
Each $A_{\xi}^{\circ}$ can be written uniquely as 
\begin{equation}
\label{eq:Axi}
\alcovxi=
\left\{x
\in V 
\left|
\begin{array}{ll}
(\alpha, x)>k_\alpha & \mbox{ for }\alpha\in I\\
(\beta, x)<k_\beta & \mbox{ for }\beta\in J
\end{array}
\right.
\right\}, 
\end{equation}
for some positive roots $I, J\subset\Phi^+$ with $|I\sqcup J|=\ell+1$, 
and $k_\alpha, k_\beta\in\bZ$ ($\alpha\in I, \beta\in J$). 
By definition, the 
facets of $\overline{\alcovxi}$ are supported by the hyperplanes 
$H_{\alpha, k_\alpha}$ ($\alpha\in I$) and 
$H_{\beta, k_\beta}$ ($\beta\in J$). 

\begin{definition}
\label{def:semialcov}
With notation as above, 
let us define the partially closed alcove $\semialcov$ by 
\begin{equation}
\label{eq:Axisemi}
\semialcov:=
\left\{x
\in V 
\left|
\begin{array}{ll}
(\alpha, x)>k_\alpha & \mbox{ for }\alpha\in I\\
(\beta, x)\leq k_\beta & \mbox{ for }\beta\in J
\end{array}
\right.
\right\}. 
\end{equation}
\end{definition}

Obviously, the interior of $\semialcov$ is $A_{\xi}^{\circ}$. Although 
$\semialcov$ is not a closure of $\alcovxi$, 
$\semialcov$ may be considered as the partial closure of $\alcovxi$. 

\begin{proposition}
\label{prop:semiclosure}
Let $\rho=\sum_{i=1}^\ell\varpi_i^\lor$. 
Then $x\in \semialcov$ if and only if 
for sufficiently small $0<\varepsilon\ll 1$, 
$x-\varepsilon\cdot\rho\in \alcovxi$, 
(that is, there exists 
$\varepsilon_0>0$ such that if $0<\varepsilon<\varepsilon_0$, then  
$x-\varepsilon\cdot\rho\in \alcovxi$). 
\end{proposition}
\begin{proof}
Straightforward. 
\end{proof}
From Proposition \ref{prop:semiclosure}, we have a partition of 
$\parallelp$. 
\begin{proposition}
\label{prop:partition}
\begin{equation}
\label{eq:partition}
\parallelp=\bigsqcup_{\xi\in\Xi}\semialcov . 
\end{equation}
\end{proposition}
\begin{proof}
It is enough to show that each $x\in\parallelp$ is contained 
in the unique $\semialcov$. Let $x\in\parallelp$. Then 
for sufficiently small $\varepsilon>0$, 
$(\alpha, x-\varepsilon\cdot\rho)\notin\bZ$ for all $\alpha\in\Phi^+$, 
hence $x-\varepsilon\cdot\rho$ is contained in the unique alcove $\alcovxi$. 
By Proposition \ref{prop:semiclosure}, $x$ is contained in the corresponding 
$\semialcov$. 
\end{proof}

\subsection{Shift operator and ``Riemann hypothesis''}
\label{subsec:rh}

Let $a, b\in\bZ$ be integers with $a\leq b$. Let us denote 
by $\A_{\Phi}^{[a,b]}$ the hyperplane arrangement 
\[
\A_{\Phi}^{[a,b]}=
\{H_{\alpha, k}\mid \alpha\in\Phi^+, k\in\bZ, a\leq k\leq b\}. 
\]
By Proposition \ref{prop:rootsys} (iii), we have 
$|\A_{\Phi}^{[a,b]}|=\frac{\ell\cdot h\cdot (b-a+1)}{2}$. 
For special cases, the characteristic polynomial 
$\chi(\A_{\Phi}^{[a,b]}, t)$ factors. 
\begin{theorem}
\label{thm:charpoly}
\begin{itemize}
\item[(i)] 
If $k\geq 0$, then 
$\chi(\A_{\Phi}^{[-k,k]}, t)=\prod_{i=1}^\ell(t-e_i-kh)$. 
\item[(ii)] 
If $k\geq 1$, then 
$\chi(\A_{\Phi}^{[1-k,k]}, t)=(t-kh)^\ell$. 
\end{itemize}
\end{theorem}
The above result had been conjectured by Edelman and Reiner \cite{ede-rei}. 
Theorem \ref{thm:charpoly} (i) was proved in 
\cite{ath-gen} by using lattice point counting techniques, 
which will be developed further in this paper. 
Theorem \ref{thm:charpoly} (ii) was proved in 
\cite{yos-char} by use of the theory of free arrangements 
(\cite{ede-rei, ot, ter-fact, ter-multi}).

For an interval $[a, b]\neq [-k, k], [1-k, k]$, the characteristic 
polynomial 
$\chi(\A_{\Phi}^{[a,b]}, t)$ does not factor in general. 
Postnikov and Stanley pose the following 
``Riemann hypothesis''. 

\begin{conjecture}
\label{conj:rh}
(\cite[Conjecture 9.14]{ps-def}) Let $a, b\in\bZ$ with $a\leq 1\leq b$. 
Suppose $a+b\geq 1$. Then every root $t\in\bC$ of the equation 
$\chi(\A_{\Phi}^{[a,b]}, t)=0$ satisfies 
$\Re t=\frac{h(b-a+1)}{2}$. 
\end{conjecture}
Conjecture \ref{conj:rh} has been proved by Stanley, Postnikov and 
Athanasiadis in \cite{ath-lin, ps-def} 
for $\Phi\in\{A_\ell, B_\ell, C_\ell, D_\ell, G_2\}$. 
We recall their results. 

Let $f:\bN\longrightarrow\bR$ be a partial function, that is, 
a function defined on a subset of $\bN$. Define the action of 
the \emph{shift operator} $S$ by 
\[
(Sf)(t)=f(t-1). 
\]
More generally, for a polynomial $P(S)=\sum_ka_kS^k$ in $S$, the action is 
defined by 
\[
(P(S)f)(t)=\sum_k a_k f(t-k). 
\]
\begin{proposition}
\label{prop:annihilate}
Let $g(S)\in\bR[S]$ and $f(t)\in\bR[t]$. Suppose $\deg f=n$. 
Then $g(S)f=0$ if and only if $(1-S)^{n+1}|g(S)$. 
\end{proposition}
\begin{proof}
First note that since $(1-S)f(t)=f(t)-f(t-1)$ is the difference operator, 
$\deg ((1-S)f)=n-1$. Suppose $(1-S)^{n+1}|g(S)$. Then by induction, 
it is easily seen that $(1-S)^{n+1}f=0$. Hence $g(S)f=0$. 

Conversely, suppose $g(S)f=0$. Consider the Taylor expansion 
of $g(S)$ at $S=1$. Set $g(S)=b_0+b_1(S-1)+b_2(S-1)^2+\dots+
b_n(S-1)^n+(S-1)^{n+1}\widetilde{g}(S)$. Since $(S-1)^{n+1}f=0$, 
we have 
\begin{equation}
\label{eq:leadingterm}
(b_0+b_1(S-1)+b_2(S-1)^2+\dots+b_n(S-1)^n)f=0. 
\end{equation}
Set $f(t)=r_0t^n+r_1t^{n-1}+\dots+r_n$ with $r_0\neq 0$. 
The coefficient of the term of degree $n$ in 
(\ref{eq:leadingterm}) is $b_0 r_0$. Hence $b_0=0$. Similarly, 
$b_1=\dots=b_n=0$, and we have $g(S)=(S-1)^{n+1}\widetilde{g}(S)$. 
\end{proof}
The shift operator can be used to express characteristic polynomials. 

\begin{theorem}
\label{thm:rhABCD}
(\cite{ath-lin, ps-def}). 
\begin{itemize}
\item[(1)] Let $n\geq 1$ and $k\geq 0$. 
For the cases $\Phi\in\{A_\ell, B_\ell, C_\ell, D_\ell\}$, 
$\chi(\A_{\Phi}^{[1-k, n+k]}, t)=\chi(\A_{\Phi}^{[1,n]}, t-kh)$. 
\item[(2)] Let $n\geq 1$. Then the characteristic polynomial 
$\chi(\A_{\Phi}^{[1,n]}, t)$ has the following expression. 
\begin{itemize}
\item[(i)] For $\Phi=A_\ell$, 
\begin{equation}
\label{eq:extendedlinialA}
\chi(\A_{A_\ell}^{[1,n]}, t)=
\left(
\frac{1+S+S^2+\dots+S^n}{1+n}
\right)^{\ell+1}t^\ell. 
\end{equation}
\item[(ii)] For $\Phi=B_\ell$ or $C_\ell$, 
\begin{equation}
\small \chi(\A_{\Phi}^{[1,n]}, t)=
\left\{
\begin{array}{ll}
\frac{4S(1+S^2+S^4+\dots+ S^{2n})^{\ell-1}(1+S^2+S^4+\dots+S^{n-1})^2}{(1+n)^{\ell+1}}t^\ell, & 
\mbox{if $n$ odd}, \\
&\\
\frac{(1+S^2+S^4+\dots +S^{2n})^{\ell-1}(1+S^2+S^4+\dots+S^{n})^2}{(1+n)^{\ell+1}}t^\ell, & 
\mbox{if $n$ even}. 
\end{array}
\right.
\end{equation}
\item[(iii)] For $\Phi=D_\ell$, 
\begin{equation}
\chi(\A_{D_\ell}^{[1,n]}, t)=
\left\{
\begin{array}{ll}
\frac{8S(1+S^2)(1+S^2+S^4+\dots +S^{2n})^{\ell-3}(1+S^2+S^4+\dots+S^{n-1})^4}{(1+n)^{\ell+1}}t^\ell, & 
\mbox{if $n$ odd}, \\
&\\
\frac{(1+S^2+S^4+\dots +S^{2n})^{\ell-3}(1+S^2+S^4+\dots+S^{n})^4}{(1+n)^{\ell+1}}t^\ell, & 
\mbox{if $n$ even}. 
\end{array}
\right.
\end{equation}
\end{itemize}
\end{itemize}
\end{theorem}
Owing to the next result, 
the above expressions implies Conjecture \ref{conj:rh} for 
$\A_\Phi^{[a, b]}$ with 
$\Phi=A_\ell, B_\ell, C_\ell$, or $D_\ell$ and $a+b\geq 2$. 

\begin{lemma}
\label{lem:polya}
(\cite[Lemma 9.13]{ps-def}) Let $f(t)\in\bC[t]$. Suppose all the roots of 
the equation $f(t)=0$ have real part equal to $a$. Let 
$g(S)\in\bC[S]$ be a polynomial such that every root of the equation $g(z)=0$ 
satisfies $|z|=1$. Then all roots of the equation 
$(g(S)f)(t)=0$ have 
real part equal to 
$a+\frac{\deg g}{2}$. 
\end{lemma}

\begin{remark}
\label{rem:remonrh}
\begin{itemize}
\item[(1)] 
The ``Riemann hypothesis'' for 
the special case $a+b=1$ is a consequence of 
Theorem \ref{thm:charpoly} (ii). 
\item[(2)] 
Conjecture \ref{conj:rh} implies the ``functional equation'' 
(\cite[(9.12)]{ps-def})
\begin{equation}
\label{eq:fe}
\chi(\A_\Phi^{[a, b]}, h(b-a+1)-t)=(-1)^\ell
\chi(\A_\Phi^{[a, b]}, t), 
\end{equation}
for $a\leq 1\leq b$ satisfying $a+b\geq 1$. 
The relation (\ref{eq:fe}) for characteristic quasi-polynomials 
will be proved later 
(Theorem \ref{thm:dualityquasichar} and 
Corollary \ref{cor:generalfuncteq}). 
\item[(3)] 
The ``functional equation '' (\ref{eq:fe}) is also valid for 
the case $[a, b]=[-k, k]$, owing to the duality of exponents 
$e_i+e_{\ell-i+1}=h$. 
\end{itemize}
\end{remark}

\subsection{Eulerian polynomial}
\label{subsec:eulerian}

The Eulerian polynomial was originally introduced by Euler 
for the purpose of describing the special value of the zeta 
function $\zeta(n)$ at negative integers $n<0$ \cite{hir-eu}. 
Currently, it plays an important role in enumerative combinatorics 
\cite{st-ec1}. 

\begin{definition}
\label{def:ascent}
For a permutation $\sigma\in\frakS_n$, define 
\[
\begin{split}
a(\sigma)
&=
\#\{i\mid 1\leq i\leq n-1, \sigma(i)<\sigma(i+1)\},\\
d(\sigma)
&=
\#\{i\mid 1\leq i\leq n-1, \sigma(i)>\sigma(i+1)\}. 
\end{split}
\]
\end{definition}
Then 
\[
A(n, k)=\#
\{\sigma\in\frakS_n\mid a(\sigma)=k-1\}, 
\]
$(1\leq k\leq n-1)$ 
is called the \emph{Eulerian number} and the generating polynomial 
\[
\eul_n(t)=
\sum_{k=1}^n A(n, k)t^k=
\sum_{\sigma\in\frakS_n}t^{1+a(\sigma)}
\]
is called the \emph{Eulerian polynomial}. It is easily seen that 
$A(n, k)=A(n, n-k+1)$. It follows immediately that 
$t^{n+1}\eul_n(\frac{1}{t})=\eul_n(t)$. 
The first eight Eulerian polynomials are as follows. 
\[
\begin{split}
\eul_1(t)&=t\\
\eul_2(t)&=t+t^2\\
\eul_3(t)&=t+4t^2+t^3\\
\eul_4(t)&=t+11t^2+11t^3+t^4\\
\eul_5(t)&=t+26t^2+66t^3+26t^4+t^5\\
\eul_6(t)&=t+57t^2+302t^3+302t^4+57t^5+t^6\\
\eul_7(t)&=t+120t^2+1191t^3+2416t^4+1191t^5+120t^6+t^7\\
\eul_8(t)&=t+247t^2+4293t^3+15619t^4+15619t^5+4293t^6+247t^7+t^8\\
\end{split}
\]
The next formula is one of the classical results concerning 
Eulerian numbers. 
\begin{theorem}
\label{thm:worpitzky}
(Worpitzky \cite{wor}, see also \cite{comtet}) 
\begin{equation}
\label{eq:classicalwor}
t^n=\sum_{k=1}^n A(n, k)
\begin{pmatrix}
t+k-1\\
n
\end{pmatrix}. 
\end{equation}
\end{theorem}
\begin{remark}
\label{rem:reformulatewor}
Using the shift operator $S$ (in \S\ref{subsec:rh}), 
the Worpitzky identity (\ref{eq:classicalwor}) can be reformulated as 
\begin{equation}
\label{eq:shiftwor}
t^n=\eul_n(S)
\frac{(t+n)(t+n-1)\cdots(t+1)}{n!}. 
\end{equation}
In \S \ref{subsec:worpart} 
we will give another proof of (\ref{eq:shiftwor}). 
The polynomial 
$\frac{(t+n)(t+n-1)\cdots(t+1)}{n!}$ is the Ehrhart polynomial 
of the fundamental alcove for the root system of type $A_{n}$. 
If we replace it with the Ehrhart quasi-polynomial of the fundamental 
alcove then we obtain similar formulae for root systems. 
(See Theorem \ref{thm:worpart} and Remark \ref{rem:worpitzkytypeA}.) 
\end{remark}

\section{Ehrhart quasi-polynomial for the fundamental alcove}

\label{sec:ehr}

\subsection{Ehrhart quasi-polynomial}
\label{subsec:ehrhart}

A convex polytope $\scP$ is a convex hull of finitely many points 
in $\bR^n$. A polytope $\scP\subset\bR^n$ is said to be 
integral (resp. rational) if all vertices of $\scP$ are 
contained in $\bZ^n$ (resp. $\bQ^n$). We denote by $\scP^{\circ}$ the 
relative interior of $\scP$. 

Let $\scP$ be a rational polytope. 
For a positive integer $q\in\bZ_{>0}$, define 
\begin{equation}
\label{eq:ehr}
\ehr_{\scP}(q)=\#(q \scP\cap\bZ^n). 
\end{equation}
Similarly, define $\ehr_{\scP^\circ}(q)=\#(q \scP^\circ\cap\bZ^n)$. 
These functions are known to be quasi-polynomials 
(\cite[Theorem 3.23]{be-ro}). (Moreover, 
the minimal period of the quasi-polynomial divides the 
least common multiple of the denominators of vertex coordinates.) 
Thus the value $\ehr_{\scP}(q)$ makes 
sense for negative $q$ and is related to $\ehr_{\scP^\circ}(q)$ 
by the following reciprocity property 
\begin{equation}
\label{eq:reciproc}
\ehr_{\scP}(-q)=(-1)^{\dim\scP}\ehr_{\scP^\circ}(q),  
\end{equation}
for $q>0$.

\subsection{Ehrhart quasi-polynomial for $\baralcov$}
\label{subsec:ehralcov}

Let $\baralcov$ be the closed fundamental alcove 
of type $\Phi$ (\S \ref{subsec:rs}). Suter computes the 
Ehrhart quasi-polynomial $\ehr_{\baralcov}(q)$ (with respect to 
the coweight lattice $\cwtl(\Phi)$) in \cite{sut} (see also 
\cite{bl-sa, ath-adv, ath-gen, hai, ktt-quasi}). 
See Example \ref{ex:ehrhartquasi} for (some 
of) the explicit formulae. Several useful conclusions may 
be summarized as follows. 
\begin{theorem}
\label{thm:suter}
(Suter \cite{sut}) 
\begin{itemize}
\item[(i)] 
The Ehrhart quasi-polynomial $\ehr_{\baralcov}(q)$ 
has the $\gcd$-property. 
\item[(ii)] 
The leading coefficient of $\ehr_{\baralcov}(q)$ is $\frac{f}{|W|}$. 
\item[(iii)] 
The minimal period $\widetilde{n}$ is 
as given in the table (Table \ref{fig:table}). 
\item[(iv)] If $q$ is relatively prime to the period $\widetilde{n}$, 
then 
\[
\ehr_{\baralcov}(q)=\frac{f}{|W|}(q+e_1)(q+e_2)\cdots(q+e_\ell). 
\]
\item[(v)] $\rad(\widetilde{n})|h$, 
where $\rad(\widetilde{n})=
\prod_{p: \mbox{\scriptsize prime}, p|\widetilde{n}}p$ 
is the radical of $\widetilde{n}$. 
\end{itemize}
\end{theorem}

\begin{table}[htbp]
\centering
{\footnotesize 
\begin{tabular}{c|c|c|c|c|c|c|c}
$\Phi$&$e_1, \dots, e_\ell$&$c_1, \dots, c_\ell$&$h$&$f$&$|W|$&$\widetilde{n}$&$\rad(\widetilde{n})$\\
\hline\hline
$A_\ell$&$1,2,\dots,\ell$&$1,1,\dots,1$&$\ell+1$&$\ell+1$&$(\ell+1)!$&$1$&$1$\\
$B_\ell, C_\ell$&$1,3,5,\dots,2\ell-1$&$1,2,2,\dots,2$&$2\ell$&$2$&$2^\ell\cdot \ell!$&$2$&$2$\\
$D_\ell$&$1,3,5,\dots,2\ell-3,\ell-1$&$1,1,1,2,\dots,2$&$2\ell-2$&$4$&$2^{\ell-1}\cdot\ell!$&$2$&$2$\\
$E_6$&$1,4,5,7,8,11$&$1,1,2,2,2,3$&$12$&$3$&$2^7\cdot 3^4\cdot 5$&$6$&$6$\\
$E_7$&$1,5,7,9,11,13,17$&$1,2,2,2,3,3,4$&$18$&$2$&$2^{10}\cdot 3^4\cdot 5\cdot 7$&$12$&$6$\\
$E_8$&$1,7,11,13,17,19,23,29$&$2,2,3,3,4,4,5,6$&$30$&$1$&$2^{14}\cdot 3^5\cdot 5^2\cdot 7$&$60$&$30$\\
$F_4$&$1,5,7,11$&$2,2,3,4$&$12$&$1$&$2^7\cdot 3^2$&$12$&$6$\\
$G_2$&$1,5$&$2,3$&$6$&$1$&$2^2\cdot 3$&$6$&$6$
\end{tabular}
}
\caption{Table of root systems.}
\label{fig:table}
\end{table}

\begin{example}
\label{ex:ehrhartquasi}
\begin{itemize}
\item[(1)] 
$\Phi=A_\ell$. 
The closed fundamental alcove $\baralcov$ is 
the convex hull of $0, \varpi_1^\lor, \dots, \varpi_\ell^\lor$, 
which is an integral simplex. Hence the period is $\widetilde{n}=1$. 
Moreover, 
\begin{equation}
\label{eq:ehrA}
\ehr_{\baralcov}(t)=
\frac{(t+1)(t+2)\cdots(t+\ell)}{\ell!}. 
\end{equation}
\item[(2)] 
$\Phi=B_\ell$ or $C_\ell$. 
The closed fundamental alcove $\baralcov$ is 
the convex hull of $0, \frac{\varpi_1^\lor}{2}, 
\varpi_2^\lor, \dots, \varpi_\ell^\lor$. 
The period is $\widetilde{n}=2$. 
\[
\ehr_{\baralcov}(t)=
\left\{
\begin{array}{ll}
\frac{(t+1)(t+3)\cdots(t+2\ell-1)}{2^{\ell-1}\cdot\ell!}, 
&\mbox{ if $t$ is odd}\\
&\\
\frac{(t+\ell)\prod_{i=1}^{\ell-1}(t+2i)}{2^{\ell-1}\cdot\ell!}, 
&\mbox{ if $t$ is even.}
\end{array}
\right.
\]
\item[(3)] 
$\Phi=D_\ell$. 
The period is $\widetilde{n}=2$. 
\[
\ehr_{\baralcov}(t)=
\left\{
\begin{array}{ll}
\frac{(t+\ell-1)\prod_{i=1}^{\ell-1}(t+2i-1)}{2^{\ell-3}\cdot\ell!}, 
&\mbox{ if $t$ is odd}\\
&\\
\frac{(t^2+2(\ell-1)+\frac{\ell(\ell-1)}{2})\cdot
\prod_{i=1}^{\ell-2}(t+2i)}{2^{\ell-3}\cdot\ell!}, 
&\mbox{ if $t$ is even.}
\end{array}
\right.
\]
\item[(4)] 
$\Phi=E_6$. 
The period is $\widetilde{n}=6$. 
\[
\ehr_{\baralcov}(t)=
\left\{
\begin{array}{ll}
\frac{(t+1)(t+4)(t+5)(t+7)(t+8)(t+11)}{2^3\cdot 3\cdot 6!}, 
&\mbox{ if $t\equiv 1, 5 \mod 6$}\\
&\\
\frac{(t+3)(t+9)(t^4+24t^3+195t^2+612t+480)}{2^3\cdot 3\cdot 6!}, 
&\mbox{ if $t\equiv 3 \mod 6$}\\
&\\
\frac{(t+2)(t+4)(t+8)(t+10)(t^2+12t+26)}{2^3\cdot 3\cdot 6!}, 
&\mbox{ if $t\equiv 2,4 \mod 6$}\\
&\\
\frac{(t+6)^2(t^4+24t^3+186t^2+504t+480)}{2^3\cdot 3\cdot 6!}, 
&\mbox{ if $t\equiv 0 \mod 6$.}
\end{array}
\right.
\]
\end{itemize}
\end{example}

Let $\Phi$ be an arbitrary root system. For a positive integer 
$q\in\bZ_{>0}$, the simplex $q \baralcov$ has $\ell+1$ facets, 
which will be denoted by 
\[
\begin{split}
F_0=&\baralcov\cap H_{\widetilde{\alpha}, q}, \\
F_1=&\baralcov\cap H_{{\alpha_1}, 0}, \\
F_2=&\baralcov\cap H_{{\alpha_2}, 0}, \\
&\vdots\\
F_\ell=&\baralcov\cap H_{{\alpha_\ell}, 0}. 
\end{split}
\]
We shall count the lattice 
points after removing a facet. 

\begin{lemma}
\label{lem:key}
Let $0\leq i\leq \ell$. 
Suppose $q\gg0$ (indeed $q>c_i$ is sufficient). Then, 
\begin{equation}
\label{eq:remove}
\#
\left(
(q \baralcov\cap \cwtl(\Phi))\smallsetminus F_i
\right)
=\ehr_{\baralcov}(q-c_i). 
\end{equation}
\end{lemma}
\begin{proof}
First we consider the case $i=0$. Let 
\[
x\in(q \baralcov\cap \cwtl(\Phi))\smallsetminus F_0. 
\]
Then 
$(\widetilde{\alpha},x)<q$. Since $(\widetilde{\alpha},x)$ is an 
integer, we have $(\widetilde{\alpha},x)\leq q-1$. 
Therefore, 
\[
(q \baralcov\cap\cwtl(\Phi))\smallsetminus F_0=
(q-1)\baralcov\cap\cwtl(\Phi).  
\]
Since $c_0=1$, the number of lattice points is $\ehr_{\baralcov}(q-c_0)$.

Next we consider the case $i=1$. By an argument similar to that in 
the case $i=0$, $(q \baralcov\cap\cwtl(\Phi))\smallsetminus F_1$ 
is described as 
\begin{multline*}
(q \baralcov\cap\cwtl(\Phi))\smallsetminus F_1 =\\
\left\{
x\in\cwtl(\Phi)
\mid
(\alpha_1, x)\geq 1, 
(\alpha_2, x)\geq 0, \dots, 
(\alpha_\ell, x)\geq 0, 
(\widetilde{\alpha}, x)\leq q
\right\}. 
\end{multline*}
Since $(\widetilde{\alpha}, \varpi_1^\lor)=c_1$, the map 
$x\longmapsto x+\varpi_1^\lor$ induces a bijection between 
\[
(q-c_1) \baralcov \cap\cwtl(\Phi)\stackrel{\simeq}{\longrightarrow} 
(q \baralcov\cap \cwtl(\Phi))\smallsetminus F_1. 
\]
Hence 
$\#\left((q \baralcov\cap \cwtl(\Phi))\smallsetminus F_1\right)
=\ehr_{\baralcov}(q-c_1)$. This completes the proof for $i=1$. 
The proof for $i\geq 2$ is similar. 
\end{proof}
Applying Lemma \ref{lem:key} repeatedly, we obtain the following. 
\begin{corollary}
\label{cor:partialremove}
Let $\{i_1, \dots, i_k\}\subset\{0, 1, 2, \dots, \ell\}$. 
Suppose $q\gg 0$ (indeed $q>c_{i_1}+\dots+c_{i_k}$ is sufficient). 
Then 
\[
\#\left((q \baralcov\cap \cwtl(\Phi))\smallsetminus 
(F_{i_1}\cup\dots\cup F_{i_k}\right)=
\ehr_{\baralcov}(q-c_{i_1}-\dots-c_{i_k}). 
\]
\end{corollary}

\begin{corollary}
\label{cor:fullremove}
(\cite{sut, ath-gen}) 
Let $q\in\bZ$. Then 
\begin{equation}
\label{eq:shrink}
\ehr_{\alcov}(q)=\ehr_{\baralcov}(q-h). 
\end{equation}
\end{corollary}
\begin{proof}
Since both sides of (\ref{eq:shrink}) are quasi-polynomials, 
it is sufficient to check the equality for $q\gg 0$. 
\[
(q \alcov)\cap\cwtl(\Phi)
=
(q \baralcov)\cap\cwtl(\Phi)\smallsetminus
\bigcup_{i=0}^\ell F_i. 
\]
Hence (\ref{eq:shrink}) follows from Corollary \ref{cor:partialremove} 
and the equality $\sum_{i=0}^\ell c_i=h$. 
\end{proof}
Finally, combining Corollary \ref{cor:fullremove} and 
the reciprocity of Ehrhart quasi-polynomials (\ref{eq:reciproc}), 
we obtain the following duality of 
$\ehr_{\baralcov}(q)$. 
\begin{corollary}
\label{cor:duality}
If $q\in\Z$, then 
\[
\ehr_{\baralcov}(q-h)=(-1)^\ell\ehr_{\baralcov}(-q). 
\]
\end{corollary}

\subsection{Characteristic quasi-polynomial}

\label{subsec:charquasi-poly}

In this section, we recall the relation between the 
Ehrhart quasi-polynomial of $\baralcov$ and the characteristic 
quasi-polynomial of the Weyl arrangement $\A_\Phi^{}$, following 
\cite{bl-sa, ath-adv, ath-gen, ktt-quasi} (which will be refined later). 
Recall the definition (\ref{eq:defparallelp}) 
of the fundamental parallelepiped 
$\parallelp=\{x\in V\mid 0<(\alpha_i, x)\leq 1, i=1, \dots, \ell\}$. 
Let $q>0$. 
Let us consider the projection 
\begin{equation}
\pi:\cwtl(\Phi)\longrightarrow \cwtl(\Phi)/q\cwtl(\Phi). 
\end{equation}
The restriction of $\pi$ to $q\parallelp$ 
induces a bijection 
\begin{equation}
\label{eq:bij}
\pi|_{q\parallelp \cap\cwtl(\Phi)}:
q\parallelp\cap\cwtl(\Phi)
\stackrel{\simeq}{\longrightarrow}
\cwtl(\Phi)/q\cwtl(\Phi). 
\end{equation}
To compute the characteristic quasi-polynomial, let us define 
$X_q$ by 
\begin{equation}
X_q
=
\cwtl(\Phi)
\smallsetminus
\bigcup_{\substack{\alpha\in\Phi^+ \\k\in\bZ}}H_{\alpha, kq}. 
\end{equation}
Then the projection $\pi$ induces a bijection between 
$q\parallelp\cap X_q$ and $M(\A_\Phi^{}; q)$. 

The set $q\parallelp\cap X_q$ is a disjoint union of 
$q\alcovxi\cap\cwtl(\Phi)$, ($\xi\in\Xi$). 
Therefore, by using the reciprocity (\ref{eq:reciproc}), we have 
\[
\begin{split}
|\parallelp\cap X_q|
&=
\frac{|W|}{f}\ehr_{\alcov}(q)\\
&=
\frac{|W|}{f}\cdot (-1)^\ell \ehr_{\baralcov}(-q). 
\end{split}
\]
(The case $\Phi=B_2$, $q=6$ is described in 
Figure \ref{fig:b2coxeter}. See Example \ref{ex:b2} for 
the notation.) 
Thus we have the following. 

\begin{proposition}
\label{prop:charquasicoxeter}
(\cite{ktt-quasi}) 
The characteristic quasi-polynomial of $\A_\Phi^{}$ is 
\[
\chi_{\quasi}(\A_\Phi^{}, q)=
\frac{|W|}{f}\cdot (-1)^\ell \ehr_{\baralcov}(-q). 
\]
\begin{figure}[htbp]
\centering
\begin{picture}(100,200)(0,0)
\thicklines

\thicklines

\color{red}

\multiput(20,40)(15,15){6}{\circle*{4}}
\multiput(20,55)(15,15){6}{\circle*{4}}
\multiput(20,70)(15,15){6}{\circle*{4}}
\multiput(20,85)(15,15){6}{\circle*{4}}
\multiput(20,100)(15,15){6}{\circle*{4}}
\multiput(20,115)(15,15){6}{\circle*{4}}

\thinlines

\normalcolor

\put(5,0){\line(0,1){150}}
\put(-10,155){\footnotesize $H_{\alpha_1,0}$}
\put(95,50){\line(0,1){150}}
\put(90,42){\footnotesize $H_{\alpha_1,q}$}

\put(-5,0){\line(1,1){130}}
\put(122,132){\footnotesize $H_{\alpha_2,0}$}

\put(-25,70){\line(1,1){130}}
\put(-32,62){\footnotesize $H_{\alpha_2,q}$}

\put(-5,110){\line(1,-1){65}}
\put(60,32){\footnotesize $H_{\widetilde{\alpha},q}$}

\put(40,155){\line(1,-1){65}}
\put(107,80){\footnotesize $H_{\widetilde{\alpha},2q}$}

\put(-20,100){\line(1,0){140}}
\put(122,97){\footnotesize $H_{\alpha_1+\alpha_2,q}$}

\multiput(5,10)(0,15){7}{\circle{3}}
\multiput(5,10)(15,15){7}{\circle{3}}

\thicklines

\color{blue}

\put(5,10){\line(0,1){15}}
\qbezier(5,10)(8.75,13.75)(12.5,17.5)
\qbezier(5,25)(8.75,21.25)(12.5,17.5)
\put(-7,15){\footnotesize $\alcov$}

\put(20,40){\line(1,1){15}}
\put(20,40){\line(0,1){30}}
\put(20,70){\line(1,-1){15}}
\put(21,52){\tiny $2\baralcov$}

\put(50,70){\line(1,1){15}}
\put(50,70){\line(-1,1){15}}
\put(35,85){\line(1,0){30}}
\put(45,77){\tiny $2\baralcov$}

\put(50,130){\line(1,-1){15}}
\put(50,130){\line(-1,-1){15}}
\put(35,115){\line(1,0){30}}
\put(43,117){\tiny $2\baralcov$}

\put(80,130){\line(0,1){30}}
\put(65,145){\line(1,1){15}}
\put(65,145){\line(1,-1){15}}
\put(67,143){\tiny $2\baralcov$}

\end{picture}
      \caption{$\chi_{\quasi}(\A_\Phi^{}, q)=4\ehr_{\baralcov}(q-4)$, 
($\Phi=B_2$, $q=6$).}
\label{fig:b2coxeter}
\end{figure}
\end{proposition}

We also have the duality of characteristic quasi-polynomial of the 
Weyl arrangement. 
\begin{corollary}
$\chi_{\quasi}(\A_\Phi^{}, q)=
(-1)^\ell\cdot\chi_{\quasi}(\A_\Phi^{}, h-q)$. 
\end{corollary}
\begin{proof}
Suppose $q\gg 0$. 
Using Corollary \ref{cor:duality} and 
Proposition \ref{prop:charquasicoxeter}, we have 
\[
\begin{split}
\chi_{\quasi}(\A_\Phi^{}, q)
&=
\frac{|W|}{f}\cdot (-1)^\ell \ehr_{\baralcov}(-q)\\
&=
\frac{|W|}{f}\cdot \ehr_{\baralcov}(q-h)\\
&=
(-1)^\ell\chi_{\quasi}(\A_\Phi^{}; h-q). 
\end{split}
\]
\end{proof}

\section{Generalized Eulerian polynomial}

\label{sec:geul}

\subsection{Definition and basic property}
\label{subsec:defbasics}

Using the linear relation (\ref{eq:linrel}) in \S \ref{subsec:rs}, 
we define the function 
$\asc, \dsc:W\longrightarrow\bZ$. 

\begin{definition}
\label{def:asc}
Let $w\in W$. Then $\asc(w)$ and $\dsc(w)\in\bZ$ are defined by 
\[
\begin{split}
\asc(w)&=\sum_{\substack{0\leq i\leq \ell\\ w(\alpha_i)>0}}c_i,\\
\dsc(w)&=\sum_{\substack{0\leq i\leq \ell\\ w(\alpha_i)<0}}c_i.
\end{split}
\]
\begin{remark}
Note that $\dsc(w)$ in this paper is equal to 
$\cdes(w)$ in \cite{lp-alc2}. 
\end{remark}
\end{definition}
Let $w_0\in W$ be the longest element. Since $w_0$ exchanges 
positive and negative roots, we have 
\begin{equation}
\label{eq:ascdsc}
\begin{split}
\asc(w_0w)=&\dsc(w)=h-\asc(w),\\
\dsc(w_0w)=&\asc(w)=h-\dsc(w). 
\end{split}
\end{equation}

\begin{lemma}
\label{lem:transl}
\begin{itemize}
\item[(1)] 
Let $w\in W$. Suppose that $w$ induces a permutation on 
$\{\alpha_0, \alpha_1, \dots, \alpha_\ell\}$. If $w(\alpha_i)=\alpha_{p_i}$, 
then $c_i=c_{p_i}$. 
\item[(2)] 
Let $w_1, w_2\in W$. If there exists $\gamma\in V$ (actually 
$\gamma\in\crtl(\Phi)$) such that $w_2\alcov=w_1\alcov +\gamma$, 
then $\asc(w_1)=\asc(w_2)$. 
\end{itemize}
\end{lemma}
\begin{proof}
(1) Applying $w$ to the linear relation (\ref{eq:linrel}), 
we have 
\begin{equation}
\label{eq:comparerel}
\sum_{i=0}^\ell c_iw(\alpha_i)=
\sum_{i=0}^\ell c_i \alpha_{p_i}=0. 
\end{equation}
Note that any $\ell$ distinct members of 
$\{\alpha_0, \alpha_1, \dots, \alpha_\ell\}$ 
are linearly independent. Therefore, the space 
of linear relations has dimension $1$. 
Both (\ref{eq:linrel}) and (\ref{eq:comparerel}) are linear 
relations with positive coefficients 
normalized in such a way that the minimal coefficient is equal to $1$ 
($c_0=1$). 
Hence (\ref{eq:linrel}) and (\ref{eq:comparerel}) are identical, which 
yields $c_i=c_{p_i}$. 

(2) Suppose $w_2\alcov=w_1\alcov +\gamma$. Each side is 
\[
\begin{split}
w_1\alcov +\gamma
=&
\left\{
x\in V
\left|
\begin{array}{ll}
(w_1\alpha_0, x)>(w_1\alpha_0, \gamma)-1,& \\
(w_1\alpha_i, x)>(w_1\alpha_i, \gamma),& i=1, \dots, \ell
\end{array}
\right.
\right\}, 
\\
w_2\alcov 
=&
\left\{
x\in V
\left|
\begin{array}{ll}
(w_2\alpha_0, x)>-1,& \\
(w_2\alpha_i, x)>0,& i=1, \dots, \ell
\end{array}
\right.
\right\}. 
\end{split}
\]
Since the supporting hyperplanes should coincide, 
we have 
\[
\{w_1\alpha_0, w_1\alpha_1, \dots, w_1\alpha_\ell\}=
\{w_2\alpha_0, w_2\alpha_1, \dots, w_2\alpha_\ell\}. 
\]
Thus (a modified version of) (1) enables us 
to deduce $\asc(w_1)=\asc(w_2)$. 
\end{proof}

\begin{definition}
\label{def:geneul}
The generalized Eulerian polynomial $\geneul_\Phi(t)$ is defined by 
\[
\geneul_\Phi(t)=\frac{1}{f}\sum_{w\in W}t^{\asc(w)}. 
\]
\end{definition}
The following proposition gives some basic properties of 
$\geneul_\Phi(t)$. We omit the proof, since they are immediate 
consequences of Theorem \ref{thm:lp} by Lam and Postnikov. 
(Direct proofs are also easy. In particular, the duality (2) is 
immediately deduced from (\ref{eq:ascdsc}).) 
\begin{proposition}
\label{prop:properties}
\begin{itemize}
\item[(1)] $\deg \geneul_\Phi(t)=h-1$. 
\item[(2)] (Duality) $t^h\cdot\geneul_\Phi(\frac{1}{t})=\geneul_\Phi(t)$. 
\item[(3)] $\geneul_\Phi(t)\in\bZ[t]$. 
\item[(4)] $\geneul_{A_\ell}(t)=\eul_\ell(t)$. 
\end{itemize}
\end{proposition}

The polynomial $\geneul_\Phi(t)$ was introduced by 
Lam and Postnikov in \cite{lp-alc2}. They proved 
that $\geneul_\Phi(t)$ can be expressed in terms of cyclotomic 
polynomials and the classical Eulerian polynomial. 
\begin{theorem}
\label{thm:lp}
(\cite[Theorem 10.1]{lp-alc2}) Let $\Phi$ be a root system of rank $\ell$. 
Then 
\begin{equation}
\label{eq:lp}
\geneul_\Phi(t)=[c_1]_t\cdot [c_2]_t\cdots [c_\ell]_t\cdot\eul_\ell(t), 
\end{equation}
where $[c]_t=\frac{t^c-1}{t-1}$. 
\end{theorem}

Let 
$A'\subset V\smallsetminus\bigcup_{\alpha\in\Phi^+, k\in\bZ} H_{\alpha, k}$ 
be an arbitrary alcove. We can write $A'=w(\alcov)+\gamma$ for 
some $w\in W$ and $\gamma\in\crtl(\Phi)$. Then let us define 
\begin{equation}
\label{eq:ascgen}
\asc(A'):=\asc(w), 
\end{equation}
which is indeed well-defined because of the translational invariance 
(Lemma \ref{lem:transl} (2)). Thus we can extend $\asc$ as a 
function on the set of all alcoves. Using this extension, 
we have another expression for 
$\geneul_\Phi(t)$. 
\begin{theorem}
\label{thm:eulerianparallelp}
\begin{equation}
\label{eq:eulerianpara}
\geneul_{\Phi}(t)=
\sum_{A'\subset\parallelp}t^{\asc(A')}
=
\sum_{\xi\in\Xi}t^{\asc(\alcovxi)}. 
\end{equation}
\end{theorem}
\begin{proof}
For any $w\in W$, there exists a unique $\gamma\in\crtl(\Phi)$ 
such that $w(\alcov)+\gamma\subset\parallelp$. This induces a 
map $\varphi:W\longrightarrow\{\alcovxi\mid\xi\in\Xi\}$. 
The map is surjective and $\#\varphi^{-1}(\alcovxi)=f$ holds for 
any alcove $\alcovxi\subset\parallelp$ (see \cite[page 99]{hum}). 
The assertion follows from the definition of $\geneul_\Phi(t)$. 
\end{proof}

\subsection{Worpitzky partition}
\label{subsec:worpart}

From the definition 
$\parallelp=\sum_{i=1}^\ell (0,1]\varpi_i^\lor$, 
\begin{equation}
q\parallelp\cap\cwtl(\Phi)=
\{t_1\varpi_1^\lor+\dots+t_\ell\varpi_\ell^\lor\mid
t_i\in\bZ, 0< t_i\leq q\}. 
\end{equation}
Hence we have 
\begin{equation}
\label{eq:qell}
\ehr_{\parallelp}(q)=
\#(q\parallelp\cap\cwtl(\Phi))=
q^\ell. 
\end{equation}
The partition (\ref{eq:partition}) 
$\parallelp=\bigsqcup_{\xi\in\Xi}\semialcov$ 
in Proposition \ref{prop:partition} induces a partition of lattice 
points, 
\begin{equation}
\label{eq:latticeworp}
q\parallelp\cap\cwtl(\Phi)=\bigsqcup_{\xi\in\Xi}
q\alcovxi\cap\cwtl(\Phi), 
\end{equation}
which we shall call the \emph{Worpitzky partition}. 

\begin{theorem}
\label{thm:worpart}
Suppose $q\gg 0$ (indeed $q>h$ is sufficient). Then 
\begin{equation}
\label{eq:worpitzkyPhi}
q^\ell=(\geneul_{\Phi}(S)\ehr_{\baralcov})(q). 
\end{equation}
\end{theorem}
Before the proof of this theorem, 
we will analyze the case of a single alcove. 
\begin{lemma}
\label{lem:alcovediamond}
Suppose $q\gg 0$ (indeed $q>h$ is sufficient). Then 
\begin{equation}
\label{eq:alcovediamond}
\#(q \semialcov\cap\cwtl(\Phi))=
\ehr_{\baralcov}(q-\asc(\alcovxi)). 
\end{equation}
\end{lemma}
\begin{proof}
In the notation of \S \ref{subsec:rs} (see (\ref{eq:Axisemi})), 
$q\semialcov$ is expressed as  
\begin{equation}
q\semialcov=
\left\{x
\in V 
\left|
\begin{array}{ll}
(\alpha, x)>q k_\alpha & \mbox{ for }\alpha\in I\\
(\beta, x)\leq q k_\beta & \mbox{ for }\beta\in J
\end{array}
\right.
\right\}. 
\end{equation}
Hence we have 
\begin{equation}
q\semialcov\cap\cwtl(\Phi)=
\left\{x
\in \cwtl(\Phi) 
\left|
\begin{array}{ll}
(\alpha, x)\geq q k_\alpha+1, & \mbox{for }\alpha\in I\\
(\beta, x)\leq q k_\beta, & \mbox{for }\beta\in J
\end{array}
\right.
\right\}. 
\end{equation}
From Corollary \ref{cor:partialremove} and the definition (\ref{eq:ascgen}) 
of $\asc(\alcovxi)$,  we have the equality (\ref{eq:alcovediamond}). 
\end{proof}

We now turn to the proof of Theorem \ref{thm:worpart}. Using the 
partition (\ref{eq:latticeworp}) and Lemma \ref{lem:alcovediamond}, 
we have 
\begin{equation}
\begin{split}
q^\ell
&=
\#(q\parallelp\cap\cwtl(\Phi))\\
&=
\sum_{\xi\in\Xi}\#(q\semialcov\cap\cwtl(\Phi))\\
&=
\sum_{\xi\in\Xi}\ehr_{\baralcov}(q-\asc(\alcovxi)). 
\end{split}
\end{equation}
Then applying Theorem \ref{thm:eulerianparallelp} and the shift operator, 
the right hand side can be written as 
$(\geneul_\Phi(S)\ehr_{\baralcov})(q)$, which completes the proof. 

\begin{remark}
\label{rem:worpitzkytypeA}
As we noted in Proposition \ref{prop:properties} (4), 
if $\Phi=A_\ell$ then the Eulerian polynomial is equal to 
the classical one. Furthermore, 
the Ehrhart polynomial is explicitly known (\ref{eq:ehrA}). 
Theorem \ref{thm:worpart} 
gives the classical Worpitzky identity (\ref{eq:shiftwor}). 
\end{remark}

\begin{example}
\label{ex:b2}
Let $\Phi=B_2$. Set the simple roots $\alpha_1, \alpha_2$ 
as in Figure \ref{fig:b2}. Then $\widetilde{\alpha}=\varpi_1$. 
Since $f=2$ and $|W|=8$, $\parallelp$ contains $4$ alcoves, say 
$\{A_{\xi}\mid\xi\in\Xi\}=\{A_{\xi_1}^\circ, A_{\xi_2}^\circ, 
A_{\xi_3}^\circ, A_{\xi_4}^\circ\}$ 
with the fundamental alcove $A_{\xi_1}^\circ=\alcov$. 
\begin{figure}[htbp]
\centering
\begin{picture}(100,150)(0,0)
\thicklines

\thicklines

\put(50,50){\circle*{3}}
\put(50,50){\vector(1,0){50}}
\put(105,45){$\alpha_1$}

\put(50,50){\vector(1,1){50}}
\put(105,95){$\widetilde{\alpha}=\varpi_1$}

\put(50,50){\vector(0,1){50}}
\put(35,105){$\varpi_2$}

\put(50,50){\vector(-1,1){50}}
\put(-14,105){$\alpha_2$}

\put(50,50){\vector(-1,-1){50}}
\put(-14,0){$\alpha_0$}

\thinlines

\put(50,100){\line(1,-1){25}}
\put(50,100){\line(1,0){50}}
\put(50,100){\line(1,1){50}}
\put(100,100){\line(0,1){50}}
\put(100,100){\line(-1,1){25}}

\put(105,120){$\parallelp$}

\put(53,72){\footnotesize $A_{\xi_1}^\circ$}
\put(70,89){\footnotesize $A_{\xi_2}^\circ$}
\put(70,107){\footnotesize $A_{\xi_3}^\circ$}
\put(84,122){\footnotesize $A_{\xi_4}^\circ$}

\end{picture}
      \caption{Root system of type $B_2$.}
\label{fig:b2}
\end{figure}
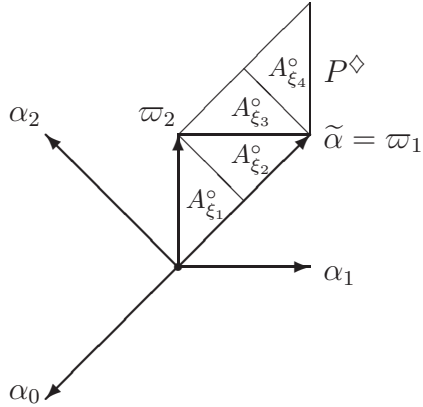
Figure \ref{fig:worpitzky} is the Worpitzky partition of 
$q\parallelp\cap\cwtl(B_2)$ for $q=6$. 
The red dots in Figure \ref{fig:worpitzky} are 
the set $6\parallelp\cap\cwtl(B_2)$, which is 
decomposed into a disjoint sum of simplices of sizes $3, 4, 4$, and $5$. 
The Eulerian polynomial is $\geneul_{B_2}(t)=t+2t^2+t^3$. 
Hence 
\[
\begin{split}
6^2
=&
\ehr_{\baralcov}(5)+
2\ehr_{\baralcov}(4)+
\ehr_{\baralcov}(3)\\
=&
((S+2S^2+S^3)\ehr_{\baralcov})(6)\\
=&
(\geneul_{B_2}(S)\ehr_{\baralcov})(6). 
\end{split}
\]

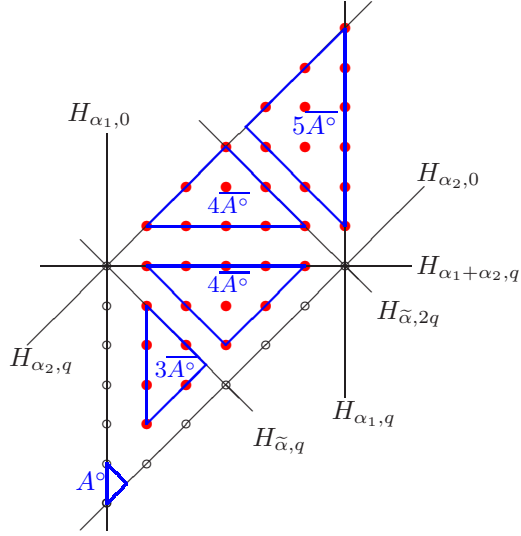
\begin{figure}[htbp]
\centering
\begin{picture}(100,200)(0,0)
\thicklines

\thicklines

\color{red}

\multiput(20,40)(15,15){6}{\circle*{4}}
\multiput(20,55)(15,15){6}{\circle*{4}}
\multiput(20,70)(15,15){6}{\circle*{4}}
\multiput(20,85)(15,15){6}{\circle*{4}}
\multiput(20,100)(15,15){6}{\circle*{4}}
\multiput(20,115)(15,15){6}{\circle*{4}}

\thinlines

\normalcolor

\put(5,0){\line(0,1){150}}
\put(-10,155){\footnotesize $H_{\alpha_1,0}$}
\put(95,50){\line(0,1){150}}
\put(90,42){\footnotesize $H_{\alpha_1,q}$}

\put(-5,0){\line(1,1){130}}
\put(122,132){\footnotesize $H_{\alpha_2,0}$}

\put(-25,70){\line(1,1){130}}
\put(-32,62){\footnotesize $H_{\alpha_2,q}$}

\put(-5,110){\line(1,-1){65}}
\put(60,32){\footnotesize $H_{\widetilde{\alpha},q}$}

\put(40,155){\line(1,-1){65}}
\put(107,80){\footnotesize $H_{\widetilde{\alpha},2q}$}

\put(-20,100){\line(1,0){140}}
\put(122,97){\footnotesize $H_{\alpha_1+\alpha_2,q}$}

\multiput(5,10)(0,15){7}{\circle{3}}
\multiput(5,10)(15,15){7}{\circle{3}}

\thicklines

\color{blue}

\put(5,10){\line(0,1){15}}
\qbezier(5,10)(8.75,13.75)(12.5,17.5)
\qbezier(5,25)(8.75,21.25)(12.5,17.5)
\put(-7,15){\footnotesize $\alcov$}

\put(20,40){\line(1,1){22.5}}
\put(20,40){\line(0,1){45}}
\put(20,85){\line(1,-1){22.5}}
\put(23,58){\scriptsize $3\baralcov$}

\put(50,70){\line(1,1){30}}
\put(50,70){\line(-1,1){30}}
\put(20,100){\line(1,0){60}}
\put(43,90){\scriptsize $4\baralcov$}

\put(50,145){\line(1,-1){30}}
\put(50,145){\line(-1,-1){30}}
\put(20,115){\line(1,0){60}}
\put(43,120){\scriptsize $4\baralcov$}

\put(95,115){\line(0,1){75}}
\put(95,115){\line(-1,1){37.5}}
\put(95,190){\line(-1,-1){37.5}}
\put(75,150){\footnotesize $5\baralcov$}

\end{picture}
      \caption{Worpitzky partition for $\Phi=B_2$ with $q=6$.}
\label{fig:worpitzky}
\end{figure}
\end{example}
We can apply the above techniques to the Shi and 
Linial arrangements. The number of lattice points 
in $q\parallelp$ minus corresponding hyperplanes are 
expressed in terms of the generalized Eulerian 
polynomial and the Ehrhart quasi-polynomial. 
(See the next section for details.) 
\begin{example}
\label{ex:b2shilinial}
Figure \ref{fig:shilinial} shows the lattice points of 
$(q\parallelp\cap\cwtl(B_2))\smallsetminus
\bigcup_{\alpha, k}(H_{\alpha, kq}\cup H_{\alpha, kq+1})$ and 
$(q\parallelp\cap\cwtl(B_2))\smallsetminus
\bigcup_{\alpha, k}H_{\alpha, kq+1}$ with $q=10$, which 
correspond to the Shi and Linial arrangements, respectively. 
In both cases, 
the decomposition can be described by using the shift 
operator, the Eulerian polynomial and the Ehrhart quasi-polynomial. 
\[
\begin{split}
\ehr_{\baralcov}(5)+
2\ehr_{\baralcov}(4)+
\ehr_{\baralcov}(3)
&=
((S^5+2S^6+S^7)\ehr_{\baralcov})(10)\\
=&
(S^4\geneul_{B_2}(S)\ehr_{\baralcov})(10). 
\end{split}
\]
\[
\begin{split}
\ehr_{\baralcov}(8)+
2\ehr_{\baralcov}(6)+
\ehr_{\baralcov}(4)
&=
((S^2+2S^4+S^6)\ehr_{\baralcov})(10)\\
=&
(\geneul_{B_2}(S^2)\ehr_{\baralcov})(10). 
\end{split}
\]
These computations will be generalized to all the root systems in 
the next section. 
\begin{figure}[htbp]
\centering
\begin{picture}(300,200)(0,0)
\thicklines

\thicklines

\color{red}

\multiput(10,20)(10,10){10}{\circle*{4}}
\multiput(10,30)(10,10){10}{\circle*{4}}
\multiput(10,40)(10,10){10}{\circle*{4}}
\multiput(10,50)(10,10){10}{\circle*{4}}
\multiput(10,60)(10,10){10}{\circle*{4}}
\multiput(10,70)(10,10){10}{\circle*{4}}
\multiput(10,80)(10,10){10}{\circle*{4}}
\multiput(10,90)(10,10){10}{\circle*{4}}
\multiput(10,100)(10,10){10}{\circle*{4}}
\multiput(10,110)(10,10){10}{\circle*{4}}


\normalcolor

\put(0,-5){\line(0,1){155}}
\put(-16,155){\tiny $H_{\alpha_1,0}$}
\put(10,-5){\line(0,1){155}}
\put(7,155){\tiny $H_{\alpha_1,1}$}

\put(100,50){\line(0,1){155}}
\put(95,42){\tiny $H_{\alpha_1,q}$}

\put(-5,-5){\line(1,1){135}}
\put(132,127){\tiny $H_{\alpha_2,0}$}
\put(-5,5){\line(1,1){135}}
\put(132,137){\tiny $H_{\alpha_2,1}$}

\put(-25,75){\line(1,1){130}}
\put(-32,67){\tiny $H_{\alpha_2,q}$}

\put(-5,105){\line(1,-1){65}}
\put(55,35){\tiny $H_{\widetilde{\alpha},q}$}
\put(-5,115){\line(1,-1){70}}
\put(65,45){\tiny $H_{\widetilde{\alpha},q+1}$}

\put(40,160){\line(1,-1){70}}
\put(108,84){\tiny $H_{\widetilde{\alpha},2q}$}
\put(45,165){\line(1,-1){70}}
\put(30,170){\tiny $H_{\widetilde{\alpha},2q+1}$}

\put(-20,100){\line(1,0){140}}
\put(122,97){\tiny $H_{\alpha_1+\alpha_2,q}$}
\put(-20,110){\line(1,0){140}}
\put(122,107){\tiny $H_{\alpha_1+\alpha_2,q+1}$}

\thinlines
\multiput(0,0)(0,10){11}{\circle{3}}
\multiput(0,0)(10,10){11}{\circle{3}}

\thicklines

\color{blue}

\put(20,40){\line(1,1){15}}
\put(20,40){\line(0,1){30}}
\put(20,70){\line(1,-1){15}}

\put(21,52){\tiny $3\baralcov$}

\put(50,70){\line(1,1){20}}
\put(50,70){\line(-1,1){20}}
\put(30,90){\line(1,0){40}}

\put(45,83){\tiny $4\baralcov$}

\put(50,140){\line(1,-1){20}}
\put(50,140){\line(-1,-1){20}}
\put(30,120){\line(1,0){40}}

\put(45,122){\tiny $4\baralcov$}

\put(90,130){\line(0,1){50}}
\put(90,130){\line(-1,1){25}}
\put(90,180){\line(-1,-1){25}}

\put(77,152){\tiny $5\baralcov$}


%
%

\thicklines

\color{red}

\multiput(210,20)(10,10){10}{\circle*{4}}
\multiput(210,30)(10,10){10}{\circle*{4}}
\multiput(210,40)(10,10){10}{\circle*{4}}
\multiput(210,50)(10,10){10}{\circle*{4}}
\multiput(210,60)(10,10){10}{\circle*{4}}
\multiput(210,70)(10,10){10}{\circle*{4}}
\multiput(210,80)(10,10){10}{\circle*{4}}
\multiput(210,90)(10,10){10}{\circle*{4}}
\multiput(210,100)(10,10){10}{\circle*{4}}
\multiput(210,110)(10,10){10}{\circle*{4}}

\thinlines

\normalcolor

\put(200,-5){\line(0,1){155}}
\put(184,155){\tiny $H_{\alpha_1,0}$}

\put(300,50){\line(0,1){155}}
\put(295,42){\tiny $H_{\alpha_1,q}$}

\put(195,-5){\line(1,1){135}}
\put(332,127){\tiny $H_{\alpha_2,0}$}

\put(175,75){\line(1,1){130}}
\put(168,67){\tiny $H_{\alpha_2,q}$}

\put(195,105){\line(1,-1){65}}
\put(255,35){\tiny $H_{\widetilde{\alpha},q}$}

\put(240,160){\line(1,-1){70}}
\put(308,84){\tiny $H_{\widetilde{\alpha},2q}$}

\put(180,100){\line(1,0){140}}
\put(322,97){\tiny $H_{\alpha_1+\alpha_2,q}$}

\thicklines

\put(209.7,-5){\line(0,1){155}}
\put(210,-5){\line(0,1){155}}
\put(210.3,-5){\line(0,1){155}}
\put(207,155){\tiny $H_{\alpha_1,1}$}

\put(194.79,5.21){\line(1,1){135}}
\put(195,5){\line(1,1){135}}
\put(195.21,4.79){\line(1,1){135}}
\put(332,137){\tiny $H_{\alpha_2,1}$}

\put(195.21,115.21){\line(1,-1){70}}
\put(195,115){\line(1,-1){70}}
\put(194.79,114.79){\line(1,-1){70}}
\put(265,45){\tiny $H_{\widetilde{\alpha},q+1}$}

\put(245.21,165.21){\line(1,-1){70}}
\put(245,165){\line(1,-1){70}}
\put(244.79,164.79){\line(1,-1){70}}
\put(230,170){\tiny $H_{\widetilde{\alpha},2q+1}$}

\put(180,110.3){\line(1,0){140}}
\put(180,110){\line(1,0){140}}
\put(180,109.7){\line(1,0){140}}
\put(322,107){\tiny $H_{\alpha_1+\alpha_2,q+1}$}

\thinlines

\multiput(200,0)(0,10){11}{\circle{3}}
\multiput(200,0)(10,10){11}{\circle{3}}

\thicklines

\color{blue}

\put(220,40){\line(1,1){20}}
\put(220,40){\line(0,1){40}}
\put(220,80){\line(1,-1){20}}
\put(221,52){\tiny $4\baralcov$}

\put(250,70){\line(1,1){30}}
\put(250,70){\line(-1,1){30}}
\put(220,100){\line(1,0){60}}
\put(245,82){\tiny $6\baralcov$}

\put(250,150){\line(-1,-1){30}}
\put(250,150){\line(1,-1){30}}
\put(220,120){\line(1,0){60}}
\put(245,122){\tiny $6\baralcov$}

\put(300,120){\line(0,1){80}}
\put(300,120){\line(-1,1){40}}
\put(300,200){\line(-1,-1){40}}
\put(280,152){\tiny $8\baralcov$}

\end{picture}
      \caption{Shi and Linial arrangements ($\Phi=B_2, q=10$).}
\label{fig:shilinial}
\end{figure}

\end{example}

\section{Shi and Linial arrangements}

\label{sec:main}

We will apply the Worpitzky partition from the previous section to 
the computation of characteristic quasi-polynomials for the Shi and Linial 
arrangements. 

\subsection{Shi arrangements}
\label{subsec:shi}

\begin{theorem}
\label{thm:shi}
Let $k\in\bZ_{>0}$. 
The characteristic quasi-polynomial $\chi_{\quasi}(\A_\Phi^{[1-k,k]}, t)$ 
of the extended Shi arrangement $\A_\Phi^{[1-k, k]}$ is equal to the 
polynomial $(t-kh)^\ell$. 
\end{theorem}
\begin{proof}
Suppose $q\gg 0$ (indeed $q>(k+1)h$ is sufficient). Set 
\begin{equation}
X_q:=
\cwtl(\Phi)\smallsetminus
\bigcup_{\substack{\alpha\in\Phi^+,\\ i, m\in\bZ,\\
1-k\leq i\leq k}}H_{\alpha, mq+i}. 
\end{equation}
We have to compute (cf. \S \ref{subsec:charquasi-poly}), 
\begin{equation}
\chi_{\quasi}(\A_\Phi^{[1-k, k]}, q)=
\#(q\parallelp\cap X_q). 
\end{equation}
Consider the Worpitzky partition 
$q\parallelp\cap\cwtl(\Phi)=
\bigsqcup_{\xi\in\Xi}(q\semialcov\cap\cwtl(\Phi))$. 
We have 
\begin{equation}
\chi_{\quasi}(\A_\Phi^{[1-k, k]}, q)=
\sum_{\xi\in\Xi}
\# (q\semialcov\cap X_q). 
\end{equation}
In the notation of Definition \ref{def:semialcov}, we have 
\begin{equation}
q\semialcov\cap X_q=
\left\{x
\in \cwtl(\Phi)
\left|
\begin{array}{ll}
(\alpha, x)\geq q k_\alpha +k+1& \mbox{ for }\alpha\in I\\
(\beta, x)\leq qk_\beta -k& \mbox{ for }\beta\in J
\end{array}
\right.
\right\}. 
\end{equation}
Hence by Corollary \ref{cor:partialremove} and Lemma 
\ref{lem:alcovediamond}, 
\begin{equation}
\#(q\semialcov\cap X_q)=
\ehr_{\baralcov}(q-kh-\asc(\alcovxi)). 
\end{equation}
Then, applying Theorem \ref{thm:eulerianparallelp}, 
\begin{equation}
\chi_{\quasi}(\A_{\Phi}^{[1-k, k]}, q)=
(\geneul_\Phi(S)\ehr_{\baralcov})(q-kh). 
\end{equation}
By Theorem \ref{thm:worpart}, the right-hand side is equal to $(q-kh)^\ell$. 
\end{proof}

By considering the case that $q$ is relatively prime to $\widetilde{n}$, 
we can conclude that the characteristic polynomial is 
\[
\chi(\A_\Phi^{[1-k,k]}, t)=(t-kh)^\ell. 
\]
This gives an alternative proof of Theorem \ref{thm:charpoly} (ii).

\subsection{Linial arrangements}
\label{subsec:linial}

In this section, we express the characteristic quasi-polynomial 
for the Linial arrangement $\A_\Phi^{[1,1+n]}$ (with $n\geq 1$) and its 
extension $\A_\Phi^{[1-k,1+n+k]}$ (with $n\geq 1, k\geq 0$) in terms 
of generalized Eulerian polynomials and Ehrhart quasi-polynomials. 

\begin{theorem}
\label{thm:linial}
Let $n\geq 1$. 
The characteristic quasi-polynomial of 
the Linial arrangement $\A_\Phi^{[1, n]}$ is 
\begin{equation}
\label{eq:linial}
\chi_{\quasi}(\A_\Phi^{[1, n]}, q)=
(\geneul_\Phi(S^{n+1})\ehr_{\baralcov})(q). 
\end{equation}
\end{theorem}
\begin{proof}
Suppose $q\gg 0$ (indeed $q>(n+1)h$ is sufficient). 
Set 
\begin{equation}
X_q:=
\cwtl(\Phi)\smallsetminus
\bigcup_{\substack{\alpha\in\Phi^+,\\ i, m\in\bZ,\\
1\leq i\leq n}}H_{\alpha, mq+i}. 
\end{equation}
In view of the bijection (\ref{eq:bij}), 
we have to compute (cf. \S \ref{subsec:charquasi-poly}) 
\begin{equation}
\chi_{\quasi}(\A_\Phi^{[1, n]}, q)=
q\parallelp\cap X_q. 
\end{equation}
By the Worpitzky partition 
$q\parallelp\cap\cwtl(\Phi)=
\bigsqcup_{\xi\in\Xi}(q\semialcov\cap\cwtl(\Phi))$, 
we have 
\begin{equation}
\chi_{\quasi}(\A_\Phi^{[1,n]}, q)=
\sum_{\xi\in\Xi}
\# (q\semialcov\cap X_q). 
\end{equation}
In the notation of Definition \ref{def:semialcov}, we have 
\begin{equation}
q\semialcov\cap X_q=
\left\{x
\in \cwtl(\Phi)
\left|
\begin{array}{ll}
(\alpha, x)\geq q k_\alpha +n+1& \mbox{ for }\alpha\in I\\
(\beta, x)\leq qk_\beta & \mbox{ for }\beta\in J
\end{array}
\right.
\right\}. 
\end{equation}
Hence by Corollary \ref{cor:partialremove} and Lemma 
\ref{lem:alcovediamond}, 
\begin{equation}
\#(q\semialcov\cap X_q)=
\ehr_{\baralcov}(q-(n+1)\asc(\alcovxi)). 
\end{equation}
Then, applying Theorem \ref{thm:eulerianparallelp}, we obtain 
\begin{equation}
\chi_{\quasi}(\A_{\Phi}^{[1,n]}, q)=
(\geneul_\Phi(S^{n+1})\ehr_{\baralcov})(q). 
\end{equation}
\end{proof}
Moreover, by an argument similar to that in the proof of 
Theorem \ref{thm:shi}, we have the following. 
\begin{theorem}
\label{thm:extendedlinial}
Let $n\geq 1$ and $k\geq 0$. 
The characteristic quasi-polynomial of 
the Linial arrangement $\A_\Phi^{[1-k, n+k]}$ is 
\begin{equation}
\label{eq:extendedlinial}
\chi_{\quasi}(\A_\Phi^{[1-k, n+k]}, q)=
\chi_{\quasi}(\A_\Phi^{[1, n]}, q-kh).
\end{equation}
\end{theorem}
Recall that by Theorem \ref{thm:ktt}, 
$\chi_{\quasi}(\A_\Phi^{[1-k, n+k]}, q)$ is a 
quasi-polynomial with $\gcd$-property. Furthermore, 
the Coxeter number $h$ is divisible by the radical 
$\rad(\widetilde{n})$ of the period $\widetilde{n}$ 
(Theorem \ref{thm:suter} (v)). 
Hence if $q$ is relatively prime to 
the period $\widetilde{n}$, then $q-kh$ is also relatively prime 
to $\widetilde{n}$. 
Hence 
$\# M(\A_\Phi^{[1, n]}, q)$ and 
$\# M(\A_\Phi^{[1,n]}, q-kh)$ are computed by using the 
same polynomial, the characteristic polynomial. Thus we obtain the 
following. 

\begin{corollary}
\label{cor:shift}
\begin{equation}
\label{eq:shift}
\chi(\A_\Phi^{[1-k, n+k]}, t)=
\chi(\A_\Phi^{[1, n]}, t-kh). 
\end{equation}
\end{corollary}

Now we have obtained two expressions of 
$\chi(\A_\Phi^{[1,n]}, t)$ for $\Phi=A_\ell$. The comparison 
of these two expressions yields a useful congruence relation 
concerning the classical Eulerian polynomial $\eul_\ell(t)$. 
Let $\Phi=A_\ell$. Set 
$g(t)=\frac{(t+1)(t+2)\cdots (t+\ell)}{\ell!}$. Then Theorem \ref{thm:linial} 
asserts that 
\begin{equation}
\label{eq:newexp}
\chi(\A_\Phi^{[1,n]}, t)=\eul_\ell(S^{n+1})g(t). 
\end{equation}
On the other hand, by formula (\ref{eq:extendedlinialA}) and the 
Worpitzky identity (\ref{eq:shiftwor}), we have another expression 
\begin{equation}
\label{eq:athana}
\chi(\A_\Phi^{[1,n]}, t)=
\left(
\frac{1+S+S^2+\dots+S^n}{1+n}
\right)^{\ell+1}
\eul_\ell(S)g(t). 
\end{equation}
By comparing the two formulae 
(\ref{eq:newexp}) and (\ref{eq:athana}) and using 
Proposition \ref{prop:annihilate}, 
we have the following congruence relation. 
\begin{proposition}
\label{prop:congruence}
Let $\ell\geq 1, m\geq 2$. Then 
\begin{equation}
\label{eq:congr}
\eul_\ell(S^{m})\equiv
\left(
\frac{1+S+S^2+\dots+S^{m-1}}{m}
\right)^{\ell+1}
\eul_\ell(S) 
\mod\ (S-1)^{\ell+1}. 
\end{equation}
\end{proposition}

\subsection{The functional equation}
\label{subsec:functeq}

Next we prove the functional equation at the level of 
characteristic quasi-polynomials. The duality of 
the generalized Eulerian polynomial plays a crucial role 
in the proof. 

\begin{theorem}
\label{thm:dualityquasichar}
\begin{equation}
\label{eq:quasifuncteqLinial}
\chi_{\quasi}(\A_\Phi^{[1,n]}, nh-t)=(-1)^\ell
\chi_{\quasi}(\A_\Phi^{[1,n]}, t). 
\end{equation}
\end{theorem}
\begin{proof}
Let $q\in\bZ$. 
We set $\geneul_\Phi(t)=\sum_{i=1}^{h-1}a_it^i$. 
Using Corollary \ref{cor:duality}, 
\[
\begin{split}
\chi_{\quasi}(\A_\Phi^{[1,n]}, nh-q)
&=
\geneul_\Phi(S^{n+1})\ehr_{\baralcov}(nh-q)\\
&=
\sum_{i=1}^{h-1}a_i\ehr_{\baralcov}(nh-q-(n+1)i)\\
&=
(-1)^\ell 
\sum_{i=1}^{h-1}a_i\ehr_{\baralcov}(q+(n+1)i-nh-h)\\
&=
(-1)^\ell 
\sum_{i=1}^{h-1}a_i\ehr_{\baralcov}(q-(n+1)(h-i)). 
\end{split}
\]
By applying the duality of $a_i=a_{h-i}$ 
(Proposition \ref{prop:properties} (2)), the right hand side 
is equal to 
\[
\begin{split}
(-1)^\ell 
\sum_{i=1}^{h-1}a_i\ehr_{\baralcov}(q-(n+1)i)
=&(-1)^\ell\geneul_\Phi(S^{n+1})\ehr_{\baralcov}(q)\\
=&(-1)^\ell\chi_{\quasi}(\A_{\Phi}^{[1,n]}, q). 
\end{split}
\]
\end{proof}
Recall that 
if $q$ is relatively prime to $\widetilde{n}$, then $mh-q$ is also 
relatively prime to $\widetilde{n}$ (Theorem \ref{thm:suter} (v)). 
By combining 
Theorem \ref{thm:charpoly}, 
Theorem \ref{thm:extendedlinial}, and 
Theorem \ref{thm:dualityquasichar}, we can formulate 
the functional equation. 
\begin{corollary}
\label{cor:generalfuncteq}
Let $a\leq 1\leq b$. Then 
\[
\chi(\A_\Phi^{[a, b]}, (b-a+1)h-t)=
(-1)^\ell\chi(\A_\Phi^{[a, b]}, t). 
\]
\end{corollary}

\subsection{Partial results on the ``Riemann hypothesis''}
\label{subsec:resultsonrh}

We will prove the ``Riemann hypothesis'' for several cases 
in $\Phi=E_6, E_7, E_8$ and $F_4$. Recall that 
\[
\rad(\widetilde{n})=
\left\{
\begin{array}{ll}
6,& \mbox{for }\Phi=E_6, E_7, F_4\\
30,& \mbox{for }\Phi=E_8. 
\end{array}
\right.
\]

\begin{theorem}
\label{thm:rh}
Let $\Phi$ be either $E_6, E_7, E_8$ or $F_4$. 
Suppose 
\[
n\equiv -1 \mod\ \rad(\widetilde{n}). 
\]
Then each root of the equation 
$\chi(\A_\Phi^{[1,n]}, t)=0$ satisfies 
$\Re=\frac{nh}{2}$. 
\end{theorem}
\begin{proof}
We give the proof only for the case $\Phi=E_6$. The proof for the other cases 
are similar. Let $n=6m-1$ ($m\in\bZ$). By Theorem \ref{thm:linial} 
$\chi_{\quasi}(\A_\Phi^{[1,6m-1]},q)=
\geneul_\Phi(S^{6m})\ehr_{\baralcov}(q)$ for $q\gg 0$. 
Set $g(t)=\frac{(t+1)(t+4)(t+5)(t+7)(t+8)(t+11)}{2^3\cdot 3\cdot 6!}$ 
and recall Example \ref{ex:ehrhartquasi} that 
if $q$ is prime to $\rad(\widetilde{n})=6$ then 
\[
\ehr_{\baralcov}(q)=g(q). 
\]
In this case $q-6k$ is also relatively prime to $6$. Hence 
\[
\chi_{\quasi}(\A_\Phi^{[1,6m-1]},q)=
\geneul_\Phi(S^{6m})g(q). 
\]
Thus we have a formula for the characteristic polynomial. 
\[
\chi(\A_\Phi^{[1,6m-1]},t)=
\geneul_\Phi(S^{6m})g(t). 
\]
Set $c(t)=[2]_t^3\cdot [3]_t$. 
Using the formula proved by Lam and Postnikov 
(Theorem \ref{thm:lp}), 
$\geneul_{E_6}(t)=c(t)\cdot\eul_6(t)$. Hence 
\[
\chi(\A_\Phi^{[1,6m-1]},t)=
c(S^{6m})\eul_6(S^{6m})g(t). 
\]
Now we employ Proposition \ref{prop:congruence}; replacing 
$S$ by $S^6$, we have 
\[
\eul_6(S^{6m})
\equiv
\left(
\frac{1+S^{6}+S^{12}+\dots+S^{6(m-1)}}{m}
\right)^7
\eul_{6}(S^6) \mod (S^6-1)^7. 
\]
Therefore, using Proposition \ref{prop:annihilate}, 
$\chi(\A_\Phi^{[1,6m-1]},t)$ can be written as 
\[
c(S^{6m})
\left(
\frac{1+S^{6}+S^{12}+\dots+S^{6(m-1)}}{m}
\right)^7
\eul_{6}(S^6)
g(t). 
\]
The first two factors are clearly cyclotomic polynomials 
in $S$. In view of Lemma \ref{lem:polya}, it is sufficient 
to check $\eul_6(S^6)g(t)$ satisfies the Riemann hypothesis. 
The explicit computation of $\eul_6(S^6)g(t)$ (up to constant 
factor) gives 
\[
29288834-8855550 t+1159185 t^2-84600 t^3+3660 t^4-90 t^5+t^6. 
\]
We can check by explicit computation that 
the six complex roots of this polynomial have common 
real part $15$. 
\end{proof}

\medskip

\noindent
\textbf{Acknowledgements:} 
The author was partially supported by 
the Grant-in-Aid for Scientific Research (C) 25400060, JSPS.

\end{document}